\documentclass[11pt]{article}
\usepackage{lscape}
\usepackage{afterpage}
\usepackage{pstricks}
\usepackage{pst-plot}
\usepackage{longtable}
\usepackage{dcolumn}
\usepackage{pst-node}
\usepackage[dvips]{graphicx}
\usepackage{amsmath}
\usepackage{amssymb}
\usepackage{amsthm}
\usepackage{multirow}
\usepackage{caption}
\usepackage{colortab}
\usepackage{color}
\usepackage{array}
\usepackage{slashbox}
\usepackage{colortbl}
\usepackage{shadbox}
\usepackage{subfigure}
\usepackage{textcomp}
\usepackage{setspace}
\usepackage{pstricks, pst-node}
\psset{arrows=->, labelsep=3pt, mnode=circle}
\usepackage{eurosym}

\newfont{\rams}{msbm10 scaled\magstep1}

\newcommand{\rio}{\mathbb{R}}

\newenvironment{resumeT}{\begin{list}{}{\setlength{\rightmargin}{\leftmargin}}\item[]
{\centering {\bf \it~~~}
\par}\item[]\ignorespaces}{\unskip\end{list}}

\newtheorem{defn}{Definition}[section]

\newtheorem{corollary}{Corollary}[section]

\newtheorem{prop}{Proposition}[section]

\newtheorem{theorem}{Theorem}[section]

\begin{document}

\title{\Large
    \textbf{Dealing with Interaction Between Bipolar Multiple\\
  Criteria Preferences in PROMETHEE Methods}
  }  
  
\author{ \hspace{0,1cm} Salvatore Corrente \thanks{Department of Economics and Business, University of Catania, Corso Italia 55, 95129  Catania, Italy, e-mails: \texttt{salvatore.corrente\string@unict.it}, \texttt{salgreco\string@unict.it}},
Jos\'e Rui Figueira \footnote{CEG-IST, Instituto Superior T\'{e}cnico, Technical University of Lisbon, Av. Rovisco Pais, 1049-001 Lisboa, Portugal, e-mail: \texttt{figueira@ist.utl.pt}}, \hspace{0,1cm} Salvatore Greco $^*$}


\date{}
\maketitle 


\addcontentsline{toc}{section}{Abstract}


\vspace{-1,5cm}

\begin{resumeT}

\textbf{Abstract:} \noindent In this paper we extend the PROMETHEE methods to the case of interacting criteria on a bipolar scale, introducing the bipolar PROMETHEE method based on the bipolar Choquet integral. In order to elicit parameters compatible with preference information provided by the Decision Maker (DM), we propose to apply the Robust Ordinal Regression (ROR). ROR takes into account simultaneously all the sets of parameters compatible with the preference information provided by the DM considering a necessary and a possible preference relation. 

\end{resumeT}

 \pagenumbering{arabic}

\vspace{0,5cm}
\section{Introduction}
In many decision making problems (for a survey on Multiple Criteria Decision Analysis (MCDA) see \cite{FigGreEhr}), alternatives are evaluated with respect to a set of criteria being not mutually preferentially independent (see \cite{wakker1989additive}). In fact, in most cases, the criteria present a certain form of positive (synergy) or negative (redundancy) interaction. For example, if one likes sport cars, maximum speed and acceleration are very important criteria. However, since in general speedy cars have also a good acceleration, giving a high weight to both criteria can over evaluate some cars. Thus, it seems reasonable to give maximum speed and acceleration considered together a weight smaller than the sum of the two weights assigned to these criteria when considered separately. In this case we have a redundancy between the criteria of maximum speed and acceleration. On the contrary, we have a synergy effect between maximum speed and price because, in general, speedy cars are also expensive and, therefore, a car which is good on both criteria is very appreciated. In this case, it seems reasonable to give maximum speed and price considered together a weight greater than the sum of the two weights assigned to these criteria when considered separately. In these cases, the aggregation of the evaluations is done by using non-additive integrals the most known of which are the Choquet integral \cite{Choquet53} and the Sugeno integral \cite{sugeno1974theory} (for a comprehensive survey on the use of non-additive integrals in MCDA see \cite{Grabisch1996,Grabisch_book_greco,Grabisch2008}).\\
In many cases, we have also to take into account that the importance of criteria may also depend on the criteria which are opposed to them. For example, a bad evaluation on aesthetics reduces the importance of maximum speed. Thus, the weight of maximum speed should be reduced when there is a negative evaluation on aesthetics. In this case, we have an antagonism effect between maximum speed and aesthetics. \\
Those types of interactions between criteria have been already taken into consideration in the ELECTRE methods \cite{FGR}. In this paper, we deal with the same problem using the bipolar Choquet integral \cite{GL1,GL2} applied to the PROMETHEE I and II methods \cite{BransMa04,BransVi85}.
 
This article extends the short paper published by the authors in \cite{Corrente2012a} with respect to which we added the description of the bipolar PROMETHEE I method, the proofs of all theorems presented in \cite{Corrente2012a} and a didactic example in which we apply the bipolar PROMETHEE methods and the Robust Ordinal Regression (ROR) \cite{greco2010robust} being a family of MCDA methods taking into account simultaneously all the sets of preference parameters compatible with the preference information provided by the Decision Maker (DM) using a necessary and a possible preference relation. \\ 
The paper is organized as follows. In the next section we recall the basic concepts of the classical PROMETHEE methods; in section 3 we introduce the bipolar PROMETHEE methods; the elicitation of preference information permitting to fix the value of the preference parameters of the model (essentially the bicapacities of the bipolar Choquet integral) is presented in section 4; in the fifth section we apply the ROR to the bipolar PROMETHEE methods; a didactic example is presented in section 6 while the last section provides some conclusions and lines for future research.

\section{The classical PROMETHEE methods}\label{PROM}
Let us consider a set of actions or alternatives $A=\left\{a,b,c,\ldots\right\}$ evaluated with respect to a set of criteria $G=\left\{g_{1},\ldots,g_{n}\right\}$, where $g_{j}:A\rightarrow\rio$, $j\in{\cal J}=\left\{1,\ldots,n\right\}$ and $|A|=m$. PROMETHEE \cite{BransMa04,BransVi85} is a well-known family of MCDA methods, among which the most known are PROMETHEE I and II, that aggregate preference information of a DM through an outranking relation. Considering for each criterion $g_{j}$ a weight $w_{j}$ (representing the importance of criterion $g_{j}$ within the family of criteria $G$), an indifference threshold $q_{j}$ (being the largest difference $d_{j}(a,b)=g_{j}(a)-g_{j}(b)$ compatible with the indifference between alternatives $a$ and $b$), and a preference threshold $p_{j}$ (being the minimum difference $d_{j}(a,b)$ compatible with the preference of $a$ over $b$), PROMETHEE methods (from now on, when we shall speak of PROMETHEE methods, we shall refer to PROMETHEE I and II) build a non decreasing function $P_{j}(a,b)$ of $d_{j}(a,b)$, whose formulation (see \cite{BransMa04} for other formulations) can be stated as follows
$$
P_{j}(a,b)=
\left\{
\begin{array}{lll}
0 & \mbox{if} & d_{j}(a,b)\leq q_{j} \\
\frac{d_{j}(a,b)-q_j}{p_{j}-q_{j}} & \mbox{if} & q_{j}<d_{j}(a,b)<p_{j} \\
1 & \mbox{if} & d_{j}(a,b)\geq p_{j} 
\end{array}
\right.
$$
\noindent The greater the value of $P_{j}(a,b)$, the greater the preference of $a$ over $b$ on criterion $g_j$. For each ordered pair of alternatives $(a,b)\in A\times A,$ PROMETHEE methods compute the value $\pi(a,b)=\sum_{j\in{\cal J}}w_{j}P_{j}(a,b)$ 
representing how much alternative $a$ is preferred to alternative $b$ taking into account the whole set of criteria. It can assume values between $0$ and $1$ and obviously the greater the value of $\pi(a,b)$, the greater the preference of $a$ over $b$. \\
In order to compare an alternative $a$ with all the other alternatives of the set $A$, PROMETHEE methods compute the negative and the positive net flow of $a$ in the following way:
$$
\phi^{-}(a)=\frac{1}{m-1}\sum_{b\in A\setminus\left\{a\right\}}\pi(b,a) \;\;\;\;\;\;\mbox{and}\;\;\;\;\;\;  \phi^{+}(a)=\frac{1}{m-1}\sum_{b\in A\setminus\left\{a\right\}}\pi(a,b).
$$
\noindent These flows represent how much the alternatives of $A\setminus\left\{a\right\}$ are preferred to $a$ and how much $a$ is preferred to the alternatives of $A\setminus\left\{a\right\}$. For each alternative $a\in A$, PROMETHEE II computes also the net flow $\phi(a)=\phi^{+}(a)-\phi^{-}(a)$.  On the basis of the positive and the negative flows, PROMETHEE I provides a partial ranking on the set of alternatives $A$, building a preference (${\cal P}^{I}$), an indifference (${\cal I}^{I}$) and an incomparability (${\cal R}^{I}$) relation. In particular: 

$$
\left\{
\begin{array}{lll}
a{\cal P}^I b & \mbox{iff} & \left\{\begin{array}{l} \Phi^{+}(a)\geq\Phi^{+}(b),\\ \Phi^{-}(a)\leq\Phi^{-}(b),\\ \Phi^{+}(a)-\Phi^{-}(a)>\Phi^{+}(b)-\Phi^{-}(b)\end{array}  \right.\\[3mm]
a{\cal I}^I b & \mbox{iff} & \left\{\begin{array}{l} \Phi^{+}(a)=\Phi^{+}(b),\\ \Phi^{-}(a)=\Phi^{-}(b)\end{array}  \right.\\[3mm]
a{\cal R}^{I}b & \mbox{otherwise} & \\
\end{array}
\right.
$$

On the basis instead of the net flows, the PROMETHEE II method provides a complete ranking on the set of alternatives $A$ defining, in a natural way, a preference (${\cal P}^{II}$) and an indifference (${\cal I}^{II}$) relation for which $a{\cal P}^{II}b$ iff $\Phi(a)>\Phi(b)$ while $a{\cal I}^{II}b$ iff $\Phi(a)=\Phi(b)$.\

\section{The bipolar PROMETHEE methods}
In order to extend the classical PROMETHEE methods to the bipolar framework, we define for each criterion $g_j$, $j\in{\cal J}$, the bipolar preference function $P_j^B: A \times A \rightarrow [-1,1], j\in {\cal J}$ in the following way: 

 \begin{equation}\label{equat}
  P_{j}^{B}(a, b) =P_j(a,b) - P_j(b,a)=
 \left\{
 \begin{array}{lll}
  P_j(a,b) & \mbox{if} & P_j(a,b) > 0 \\[2mm]
  - P_j(b,a) & \mbox{if} & P_j(a,b) = 0
 \end{array}
 \right.
 \end{equation}

\noindent It is straightforward proving that $P_{j}^{B}(a,b)=-P_{j}^{B}(b,a)$ for all $j\in{\cal J}$ and for all pairs $(a,b)\in A\times A.$

In this section we propose to aggregate the bipolar vector $P^{B}(a,b)=\left[P_{1}^{B}(a,b),\ldots,P_{n}^{B}(a,b)\right]$ through the bipolar Choquet integral. \\
The bipolar Choquet integral is based on a bicapacity \cite{GL1,GL2}, being a function $\hat\mu: P({\cal J})\rightarrow[-1,1]$, where $P({\cal J})=\left\{(C,D): C,D\subseteq {\cal J} \mbox{ and } C \cap D=\emptyset \right\}$, such that
\begin{itemize}
	\item $\hat\mu(\emptyset, {\cal J})=-1, \hat\mu({\cal J},\emptyset)=1,$ $\hat\mu(\emptyset, \emptyset)=0$ (boundary conditions),
	\item for all $(C,D),(E,F)\in P({\cal J})$, if $C\subseteq E$ and $D\supseteq F$, then $\hat\mu(C,D)\leq\hat\mu(E,F)$ (monotonicity condition). 
\end{itemize}

\noindent According to \cite{GF2003,GrecoMatarazzoSlowinski02}, we consider the following expression for a bicapacity $\hat\mu$:

\begin{equation}\label{bicapacityFG}
\hat\mu(C,D)=\mu^{+}(C,D)-\mu^{-}(C,D), \mbox{\;\;\;for all } (C,D)\in P({\cal J})
\end{equation}

\noindent where $\mu^{+},\mu^{-}:P({\cal J})\rightarrow\left[0,1\right]$ such that:

\begin{equation}\label{bound_plus}
\mu^{+}({\cal J},\emptyset)=1, \;\;\;\;\;\; \mu^{+}(\emptyset,B)=0, \;\forall B\subseteq{\cal J},
\end{equation}

\begin{equation}\label{bound_minus}
\mu^{-}(\emptyset,{\cal J})=1, \;\;\;\;\;\; \mu^{-}(B,\emptyset)=0, \;\forall B\subseteq{\cal J},
\end{equation}

\begin{equation}\label{M_miu_p}
\left.
\begin{array}{l}
\mu^+(C,D) \leq \mu^+(C\cup\left\{j\right\},D),\;\;\;\; \forall (C\cup\left\{j\right\},D)\in P({\cal J}), \;\forall j\in{\cal J},\\
\mu^+(C,D) \geq \mu^+(C,D\cup\left\{j\right\}),\;\;\;\; \forall (C,D\cup\left\{j\right\})\in P({\cal J}), \;\forall j\in{\cal J}\\
\end{array}
\right\}
\end{equation}

\begin{equation}\label{M_miu_m}
\left.
\begin{array}{l}
\mu^-(C,D) \leq \mu^-(C,D\cup\left\{j\right\}),\;\;\;\; \forall (C,D\cup\left\{j\right\})\in P({\cal J}), \;\forall j\in{\cal J},\\
\mu^-(C,D) \geq \mu^-(C\cup\left\{j\right\},D),\;\;\;\; \forall (C\cup\left\{j\right\},D)\in P({\cal J}), \;\forall j\in{\cal J}\\
\end{array}
\right\}
\end{equation}

\noindent Let us observe that (\ref{M_miu_p}) are equivalent to the constraint

$$ \mu^+(C,D)\leq\mu^+(E,F), \;\;\;\mbox{for all} \;\;\;(C,D), (E,F)\in P({\cal J})\;\;\; \mbox{such that} \;\;\;C\subseteq E \;\;\mbox{and}\;\;D\supseteq F,$$

\noindent while (\ref{M_miu_m}) are equivalent to the constraint

$$ \mu^-(C,D)\leq\mu^-(E,F), \;\;\;\mbox{for all} \;\;\;(C,D), (E,F)\in P({\cal J})\;\;\; \mbox{such that} \;\;\;C\supseteq E \;\;\mbox{and}\;\;D\subseteq F.$$



The interpretation of the functions $\mu^{+}$ and $\mu^{-}$ is the following. Given the pair $(a,b)\in A\times A$, let us consider $(C,D)\in P({\cal J})$ where $C$ is the set of criteria expressing a preference of $a$ over $b$ and $D$ the set of criteria expressing a preference of $b$ over $a$. In this situation, $\mu^{+}(C,D)$ represents the importance of criteria from $C$ when criteria from $D$ are opposing them, and $\mu^{-}(C,D)$ represents the importance of criteria from $D$ opposing $C$. Consequently, $\hat\mu(C,D)$ represents the balance of the importance of $C$ supporting $a$ and $D$ supporting $b$.

Given $(a,b)\in A\times A$, the bipolar Choquet integral of $P^B(a,b)$ with respect to the bicapacity $\hat\mu$ can be written as follows

$$Ch^{B}(P^{B}(a, b), \hat{\mu})=\int_0^1\hat\mu(\{j\in{\cal J}:P_{j}^{B}(a,b)>t\},\{j\in{\cal J}:P_{j}^{B}(a,b)<-t\})dt,$$

\noindent while the bipolar comprehensive positive preference of $a$ over $b$ and the comprehensive negative preference of $a$ over $b$ with respect to the bicapacity $\hat\mu$ are respectively:

$$Ch^{B+}(P^{B}(a, b), \hat{\mu})=\int_0^1\mu^{+}(\{j\in{\cal J}:P_{j}^{B}(a,b)>t\},\{j\in{\cal J}:P_{j}^{B}(a,b)<-t\})dt,$$

$$Ch^{B-}(P^{B}(a, b), \hat{\mu})=\int_0^1\mu^{-}(\{j\in{\cal J}:P_{j}^{B}(a,b)>t\},\{j\in{\cal J}:P_{j}^{B}(a,b)<-t\})dt,$$

\noindent where $\mu^{+}$ and $\mu^{-}$ have been defined before.

From an operational point of view, the bipolar aggregation of $P^{B}(a,b)$ can be computed as follows:
for all the criteria $j \in \cal{J}$, the absolute values of these
preferences should be re-ordered in a non-decreasing way, as follows: $\vert P_{(1)}^{B}(a,b)  \vert \leq \vert P_{(2)}^{B}(a,b) \vert \leq
 \ldots \leq \vert P_{(j)}^{B}(a,b) \vert \leq \ldots \leq \vert P_{(n)}^{B}(a,b) \vert$.

\noindent The bipolar Choquet integral of $P^B(a,b)$ with respect to the bicapacity $\hat\mu$ can now be determined:

 \begin{equation}\label{bipolarcomprehensive}
 Ch^{B}(P^{B}(a, b), \hat{\mu}) = \sum_{j \in {\cal{J}}^{>}} \vert P_{(j)}^{B}(a, b) \vert
 \Big[ \hat{\mu}\left(C_{(j)}(a,b), D_{(j)}(a,b)\right)  - \hat{\mu}\left(C_{(j+1)}(a,b), D_{(j+1)}(a,b)\right) \Big]
 \end{equation}

\noindent where $P^B(a, b) = \Big[P_{j}^{B}(a,b), \; j \in{\cal{J}}\Big]$, ${\cal{J}}^{>} = \{ j \in {\cal{J}} \; : \; \vert P_{(j)}^{B}(a,b) \vert > 0\}$, $C_{(j)}(a,b) = \{ i \in {\cal{J}}^{>} \; : \; P_{i}^{B}(a,b) \geq \vert P_{(j)}^{B}(a, b) \vert
\}$, $D_{(j)}(a,b) = \{ i \in {\cal{J}}^{>} \; : \; - P_{i}^{B}(a,b) \geq \vert P_{(j)}^{B}(a, b) \vert\}$ and $C_{(n+1)}(a,b)=D_{(n+1)}(a,b)=\emptyset$.\\
We give also the following definitions:

 \begin{equation}\label{bipolarpositive}
 Ch^{B+}(P^{B}(a, b), {\mu}^{+}) = \sum_{j \in {\cal{J}}^{>}} \vert P_{(j)}^{B}(a, b) \vert
 \Big[ {\mu}^{+}\left(C_{(j)}(a,b), D_{(j)}(a,b)\right)  - {\mu}^{+}\left(C_{(j+1)}(a,b), D_{(j+1)}(a,b)\right) \Big],
 \end{equation}

 \begin{equation}\label{bipolarnegative}
 Ch^{B-}(P^{B}(a, b), {\mu}^{-}) = \sum_{j \in {\cal{J}}^{>}} \vert P_{(j)}^{B}(a, b) \vert
 \Big[ {\mu}^{-}\left(C_{(j)}(a,b), D_{(j)}(a,b)\right)  - {\mu}^{-}\left(C_{(j+1)}(a,b), D_{(j+1)}(a,b)\right) \Big].
 \end{equation}

\noindent $Ch^B(P^B(a,b), \hat{\mu})$ gives the comprehensive
preference of $a$ over $b$ and it is equivalent to $\pi(a,b) -
\pi(b,a) =  P^C(a,b)$ in the classical PROMETHEE method while $Ch^{B+}(P^B(a,b), \mu^+)$ and $Ch^{B-}(P^B(a,b), \mu^-)$ give, respectively, how much $a$ outranks $b$ (considering the reasons in favor of $a$) and how much $a$ is outranked by $b$ (considering the reasons against $a$).\\
From the definitions above, it is straightforward proving that, for all $a,b\in A,$ 

\begin{equation}\label{bipolarpref}
Ch^{B}(P^{B}(a, b), \hat{\mu})=Ch^{B+}(P^{B}(a, b), {\mu}^{+})-Ch^{B-}(P^{B}(a, b), {\mu}^{-})
\end{equation}

\noindent Using equations (\ref{bipolarcomprehensive}), (\ref{bipolarpositive}) and (\ref{bipolarnegative}), we can define for each alternative $a\in A$ the bipolar positive flow, the bipolar negative flow and the bipolar net flow as follows:

\begin{equation}\label{pos_flow}
 {\phi}^{B+}(a) = \frac{1}{m-1} \sum_{b \in A\setminus\left\{a\right\}}Ch^{B+}(P^{B}(a,b), {\mu}^{+})
\end{equation}

\begin{equation}\label{neg_flow}
 {\phi}^{B-}(a) = \frac{1}{m-1} \sum_{b \in A\setminus\left\{a\right\}}Ch^{B-}(P^{B}(a,b), {\mu}^{-})
\end{equation}

\begin{equation}\label{net_flow}
 {\phi}^{B}(a) = \frac{1}{m-1} \sum_{b \in A\setminus\left\{a\right\}}Ch^{B}(P^{B}(a,b), \hat{\mu})
\end{equation}

\noindent By equation (\ref{bipolarpref}), it follows that ${\phi}^{B}(a)={\phi}^{B+}(a)-{\phi}^{B-}(a)$ for each $a\in A$.

Analogously to the classical PROMETHEE I and II methods, using the positive, the negative and the net bipolar flows we propose the bipolar PROMETHEE I and the bipolar PROMETHEE II methods. Given a pair of alternatives $(a,b)\in A\times A$, the bipolar PROMETHEE I method defines a partial order on the set of alternatives $A$ considering a preference (${\cal P}_{B}^I$), an indifference (${\cal I}_{B}^I$) and an incomparability (${\cal R}_{B}^I$) relation defined as follows:

$$
\left\{
\begin{array}{lll}
a{\cal P}_{B}^I b & \mbox{iff} & \left\{\begin{array}{l} \Phi^{B+}(a)\geq\Phi^{B+}(b),\\ \Phi^{B-}(a)\leq\Phi^{B-}(b),\\ \Phi^{B+}(a)-\Phi^{B-}(a)>\Phi^{B+}(b)-\Phi^{B-}(b)\end{array}  \right.\\[3mm]
a{\cal I}_B^I b & \mbox{iff} & \left\{\begin{array}{l} \Phi^{B+}(a)=\Phi^{B+}(b),\\ \Phi^{B-}(a)=\Phi^{B-}(b)\end{array}  \right.\\[3mm]
a{\cal R}_B^{I}b & \mbox{otherwise} & \\
\end{array}
\right.
$$

\noindent Given a pair of alternatives $(a,b)\in A\times A$, the bipolar PROMETHEE II method provides, instead, a complete order on the set of alternatives $A$, defining the a preference (${\cal P}_{B}^{II}$) and an indifference $({\cal I}_{B}^{II})$ relations as follows: $aP_{B}^{II}b$ iff $\Phi^{B}(a)>\Phi(b)$, while $aI_{B}^{II}b$ iff $\Phi^{B}(a)=\Phi^{B}(b)$.

\subsection{Symmetry conditions}
Because $Ch^B(P^B(a,b), \hat{\mu})$ is equivalent to $\pi(a,b) - \pi(b,a) =  P^C(a,b)$ in the classical PROMETHEE method, it is reasonable expecting that, for all $a,b\in A$, $Ch^B(P^B(a,b), \hat{\mu})=-Ch^B(P^B(b,a),\hat{\mu})$. The following Proposition gives conditions to satisfy such a requirement: 

\begin{prop}\label{Sym_1}\textit{$Ch^{B}(P^{B}(a, b),\hat{\mu}) =
-Ch^{B}(P^{B}(b, a),\hat{\mu})$ ~~ for all possible ~~ $a, b$, ~~ iff ~~
$\hat{\mu}(C, D) = - \hat{\mu}(D, C)$ ~~  for each ~~ $(C, D) \in
P({\cal{J}})$}.
\end{prop}

\begin{proof}
Let us prove that if  $\hat{\mu}(C,D)=
-\hat{\mu}(D,C)$ for each $(C, D) \in
P({\cal{J}})$, then $Ch^B(P^B(a,b), \hat{\mu})= -
Ch^B(P^B(b,a), \hat{\mu})$. As noticed, $P_{j}^{B}(a,b)= -
P_{j}^{B}(b,a)$ for all $j \in \cal{J}$, and consequently $\vert P_{(j)}^{B}(a,b) \vert = \vert -P_{(j)}^{B}(b,a) \vert = \vert
  P_{(j)}^{B}(b,a) \vert$ for all $j\in{\cal J}.$

\noindent From this, it follows that:

\begin{itemize}
\item[$(\alpha)$] ${\displaystyle C_{(j)}(a,b) = \{ i \in {\cal{J}}^{>} \; : \; P_{i}^{B}(a,b) \geq  \vert
P_{(j)}^{B}(a,b) \vert  \} = \{ i \in {\cal{J}}^{>} \; : \;
-P_{i}^{B}(b,a) \geq \vert P_{(j)}^{B}(b,a) \vert \} = }$ \\
${\displaystyle = D_{(j)}(b,a)}$;
\item[$(\beta)$] ${\displaystyle D_{(j)}(a,b) = \{ i \in {\cal{J}}^{>} \; : \; - P_{i}^{B}(a,b) \geq  \vert
P_{(j)}^{B}(a,b) \vert  \} = \{ i \in {\cal{J}}^{>} \; : \;
P_{i}^{B}(b,a) \geq \vert P_{(j)}^{B}(b,a) \vert \}=}$ \\
${\displaystyle = C_{(j)}(b,a)}$.
\end{itemize}

From $(\alpha)$ and $(\beta)$ we have that

\begin{itemize}
\item[$(\gamma)$] ${\displaystyle Ch^B(P^B(a,b), \hat{\mu}) =}$ \\
=${\displaystyle \sum_{j \in {\cal{J}}^{>}} \vert P_{(j)}^{B}(a,b)
\vert \Big[\hat{\mu}(C_{(j)}(a,b), D_{(j)}(a,b)) -
\hat{\mu}(C_{(j+1)}(a,b), D_{(j+1)}(a,b))\Big] = }$ \\
${\displaystyle = \sum_{j \in {\cal{J}}^{>}} \vert
P_{(j)}^{B}(b,a) \vert \Big[\hat{\mu}(D_{(j)}(b,a), C_{(j)}(b,a))
- \hat{\mu}(D_{(j+1)}(b,a), C_{(j+1)}(b,a)) \Big]}$.
\end{itemize}

\noindent Since $\hat{\mu}(C,D) =- \hat{\mu}(D,C),\; \forall (C,D)\in P({\cal J})$, from $(\gamma)$ we
have that,

\begin{itemize}
\item[$(\delta)$] ${\displaystyle Ch^B(P^B(b,a), \hat{\mu}) =}$ \\
${\displaystyle = \sum_{j \in {\cal{J}}^{>}} \vert
P_{(j)}^{B}(b,a)
\vert \Big[\hat{\mu}(C_{(j)}(b,a), D_{(j)}(b,a)) - \hat{\mu}(C_{(j+1)}(b,a), D_{(j+1)}(b,a))\Big] = }$ \\
${\displaystyle = \sum_{j \in {\cal{J}}^{>}} \vert
P_{(j)}^{B}(b,a) \vert \Big[ - \hat{\mu}(D_{(j)}(b,a),
C_{(j)}(b,a)) + \hat{\mu}(D_{(j+1)}(b,a), C_{(j+1)}(b,a)) \Big]}$ \\
${\displaystyle = -Ch^B(P^B(a, b), \hat{\mu})}.$
\end{itemize}

Let us now prove that if $Ch^B(P^B(a,b), \hat{\mu})= -
Ch^B(P^B(b,a), \hat{\mu})$, then $\hat{\mu}(C,D) = -
\hat{\mu}(D,C)$. Let us consider the pair $(a,b)$ such that,

\begin{equation}
P_{j}^{B}(a,b)=
\left\{
\begin{array}{lll}
1  & \mbox{if} & j\in C\\
-1 & \mbox{if} & j\in D\\
0  &           & \mbox{otherwise} \\
\end{array}
\right.
\end{equation}

In this case we have that $Ch^B(P^B(a,b), \hat{\mu})=
\hat{\mu}(C,D)$ and $Ch^B(P^B(b,a), \hat{\mu})= \hat{\mu}(D,C)$.
Thus if $Ch^B(P^B(a,b), \hat{\mu})=-Ch^B(P^B(b,a), \hat{\mu})$,
by $(iv)$ we obtain that $\hat{\mu}(C,D) = - \hat{\mu}(D,C)$ and
the proof is concluded. 
\end{proof}

\noindent Analogously, because $Ch^{B+}(P^B(a,b),\mu^+)$ represents how much $a$ outranks $b$ and $Ch^{B-}(P^B(b,a),\mu^-)$ represents how much $b$ is outranked by $a$, it is reasonable expecting that $Ch^{B+}(P^B(a,b),\mu^+)$=$Ch^{B-}(P^B(b,a),\mu^-)$. Sufficient and necessary conditions to get this equality are given by the following Proposition.

\begin{prop}\label{Salvo} 
\noindent \textit{$Ch^{B+}(P^{B}(a, b),{\mu}^{+}) =
Ch^{B-}(P^{B}(b, a),{\mu}^{-})$ ~~ for all possible ~~ $a, b$, ~~ iff ~~
${\mu}^{+}(C, D) = {\mu}^{-}(D, C)$ ~~  for each ~~ $(C, D) \in
P({\cal{J}})$}.
\end{prop}
\begin{proof}
Analogous to Proposition \ref{Sym_1}.
\end{proof}

\noindent Reminding equation (\ref{bipolarpref}), the Corollary follows.

\begin{corollary}
\textit{$Ch^{B}(P^{B}(a, b),\hat{\mu}) =-
Ch^{B}(P^{B}(b, a),\hat{\mu})$ ~~ for all possible ~~ $a, b$, ~~ if ~~
${\mu}^{+}(C, D) = {\mu}^{-}(D, C)$ ~~  for each ~~ $(C, D) \in
P({\cal{J}})$}.
\end{corollary}
\begin{proof}
This can be seen as a Corollary both of Proposition \ref{Sym_1} and Proposition \ref{Salvo}. In fact, 
\begin{itemize}
\item ${\mu}^{+}(C, D) = {\mu}^{-}(D, C)$ for each $(C, D) \in
P({\cal{J}})$ implies that $\hat\mu(C,D)=-\hat\mu(D,C)$ for each $(C, D) \in P({\cal{J}})$, and by Proposition \ref{Sym_1}, it follows the thesis.
\item ${\mu}^{+}(C, D) = {\mu}^{-}(D, C)$  for each $(C, D) \in
P({\cal{J}})$ implies that $Ch^{B+}(P^{B}(a, b),{\mu}^{+}) =
Ch^{B-}(P^{B}(b, a),{\mu}^{-})$ (by Proposition \ref{Salvo}) and from this it follows obviously the thesis by equation (\ref{bipolarpref}). 
\end{itemize}
\end{proof}

\subsection{The 2-additive decomposable bipolar PROMETHEE methods}
\noindent As seen in the previous section, the use of the bipolar Choquet integral is based on a
bicapacity which assigns numerical values to each element
$P({\cal{J}})$. Let us remark that the number of elements of
$P({\cal{J}})$ is $3^{n}$. This means that the definition of a
bicapacity requires a rather huge and unpractical number of
parameters. Moreover, the interpretation of these parameters is
not always simple for the DM. Therefore, the use of the bipolar
Choquet integral in real-world decision-making problems requires
some methodology to assist the DM in assessing the preference
parameters (bicapacities).  Several studies dealing with the determination of the
relative importance of criteria were proposed in MCDA (see e.g. \cite{RoyMousseau}). The question of
the interaction between criteria was also studied in the context
of MAUT methods \cite{MarichalRo00}. \\
In the following we consider
only the $2$-additive bicapacities \cite{GL1,fujimoto2004new}, being a particular class of bicapacities.

\subsection{Defining a manageable and meaningful bicapacity measure}

\noindent According to \cite{GF2003}, we give the following decomposition of the functions $\mu^{+}$ and $\mu^{-}$ previously defined:

\begin{defn}\label{mupiumeno}
\hspace{0,1cm}
\begin{itemize}
\item ${\displaystyle \mu^{+}(C, D) = \sum_{j \in C}a^{+}(\{ j\}, \emptyset) +  \sum_{\{j,k \} \subseteq C}a^{+}(\{j, k\}, \emptyset) +
  \sum_{j \in C, \; k \in D}a^{+}(\{ j\}, \{ k \}) }$
\item ${\displaystyle \mu^{-}(C, D) = \sum_{j \in D}a^{-}(\emptyset, \{ j\}) +  \sum_{\{j,k \} \subseteq D}a^{-}(\emptyset, \{j, k\}) +
  \sum_{j \in C, \; k \in D}a^{-}(\{ j\}, \{ k \}) }$
\end{itemize}
\end{defn}

\noindent The interpretation of each $a^{\pm}(\cdot)$ is the
following:

\begin{itemize}
\item $a^+(\{ j\}, \emptyset)$, represents the power of criterion $g_j$ by itself; this value is always non negative;
\item $a^+(\{j, k\}, \emptyset)$, represents the interaction between $g_j$ and $g_k$,
when they are in favor of the preference of $a$ over $b$; when its
value is zero there is no interaction; on the contrary, when the
value is positive there is a synergy effect when putting together
$g_j$ and $g_k$; a negative value means that the two criteria are
redundant;
\item $a^+(\{ j\}, \{ k \})$, represents the power of criterion $g_k$ against criterion $g_j$, when criterion
$g_j$ is in favor of $a$ over $b$ and $g_k$ is against to the preference of $a$
over $b$; this leads always to a reduction or no effect on the
value of $\mu^+$ since this value is always non-positive.
\end{itemize}

 An analogous interpretation can be applied to the values $a^-(\emptyset, \{
j\})$, $a^-(\emptyset, \{j, k\})$, and $a^-(\{ j \}, \{ k \})$.

 In what follows, for the sake of simplicity, we will use
$a_{j}^+$, $a_{jk}^+$, $a_{j \vert k}^+$ instead of $a^+(\{ j\},
\emptyset)$, $a^+(\{j, k\}, \emptyset)$ and $a^+(\{ j\}, \{
k \})$, respectively and $a_{j}^-$, $a_{jk}^-$, $a_{j \vert
k}^-$ instead of $a^-(\emptyset, \{ j\})$, $a^-(\emptyset, \{j,
k\})$ and $a^-(\{ j \}, \{ k \})$, respectively.\\
In this way, the bicapacity $\hat\mu$, decomposed using $\mu^+$ and $\mu^-$ of Definition \ref{mupiumeno}, has the following expression:

\begin{eqnarray*}
\hat{\mu}(C,D) & = & {\mu}^{+}(C,D)-{\mu}^{-}(C,D)=\\
 & = & \sum_{j \in C}a^{+}_{j} - 
\sum_{j \in D}a^{-}_{j} + \sum_{\{j, k \} \subseteq C}a^{+}_{jk} -
\sum_{\{j, k \} \subseteq D}a^{-}_{jk} + \sum_{j \in C, \; k \in
D}a^{+}_{j \vert k} - \sum_{j \in C, \; k \in
D}a^{-}_{j \vert k}
\end{eqnarray*}

\noindent We call such a bicapacity $\hat\mu$, a \textit{ 2-additive decomposable bicapacity }. (An analogous decomposition has been proposed directly for $\hat\mu$ without considering $\mu^{+}$ and $\mu^{-}$ in \cite{Muro}). \\
Considering these decompositions for the functions $\mu^+$ and $\mu^-$, the monotonicity conditions (\ref{M_miu_p}), (\ref{M_miu_m}) and the boundary conditions (\ref{bound_plus}), (\ref{bound_minus}) have to be expressed in function of the parameters $a_j^{+}$, $a_{jk}^{+}$, $a_{j\vert k}^{+}$, $a_j^{-}$, $a_{jk}^{-}$, $a_{j\vert k}^{-}$ as follows:

\noindent {\bf{Monotonicity conditions}}
\begin{enumerate}
\item[1)] $\mu^{+}(C, D) \leq \mu^{+}(C \cup \{ j \}, D)$, $\; \; \forall \;
j \in {\cal{J}}, \; \forall  (C \cup \{ j \}, D) \in P({\cal{J}})$
 \[
  {\Leftrightarrow\displaystyle a^{+}_{j} + \sum_{k \in C}a^{+}_{jk} + \sum_{k \in D}a^{+}_{j \vert
  k} \geq 0, \; \; \forall \; j \in {\cal{J}}, \; \forall  (C \cup \{ j \}, D) \in
  P({\cal{J}})} 
 \]
\item[2)] $\mu^{+}(C, D) \geq \mu^{+}(C, D \cup \{ j \}), \; \; \forall \;
j \in {\cal{J}}, \; \forall  (C, D \cup \{ j \}) \in P({\cal{J}})$
\[ 
  {\Leftrightarrow\displaystyle  \sum_{k \in C}a^{+}_{k \vert
  j} \leq 0, \; \; \forall \; j \in {\cal{J}}, \; \forall  (C, D \cup \{ j \}) \in
  P({\cal{J}})}
 \]
being already satisfied because $a^{+}_{k \vert j}\leq 0$, $\forall k,j\in {\cal J}, k\neq j.$ \\
\end{enumerate}

\begin{enumerate}
\item[3)] $\mu^{-}(C, D) \leq \mu^{-}(C, D \cup \{ j \})$, $\; \; \forall \;
j \in {\cal{J}}, \; \forall  (C , D\cup \{ j \}) \in P({\cal{J}})$
 \[  
  {\Leftrightarrow\displaystyle a^{-}_{j} + \sum_{k \in D}a^{-}_{jk} + \sum_{k \in C}a^{-}_{k \vert
  j} \geq 0, \; \; \forall \; j \in {\cal{J}}, \; \forall  (C , D\cup \{ j \}) \in
  P({\cal{J}})}
 \]
\item[4)] $\mu^{-}(C, D) \geq \mu^{-}(C\cup \{ j \}, D ), \; \; \forall \;
j \in {\cal{J}}, \; \forall  (C\cup \{ j \}, D ) \in P({\cal{J}})$
\[  
  {\Leftrightarrow\displaystyle  \sum_{k \in D}a^{-}_{j \vert
  k} \leq 0, \; \; \forall \; j \in {\cal{J}}, \; \forall  (C\cup \{ j \}, D ) \in
  P({\cal{J}})} 
 \]
being already satisfied because $a^{-}_{j \vert k}\leq 0$, $\forall j,k\in {\cal J}, j\neq k.$ \\
\end{enumerate}

\noindent Conditions $1)$, $2)$, $3)$ and $4)$ ensure the
monotonicity of the bi-capacity, $\hat{\mu}$, on $\cal{J}$,
obtained as the difference of $\mu^+$ and
$\mu^-$, that is,

\[
\forall \; \; (C,D), \; (E, F) \; \in \; P({\cal{J}}) \; \; \;
\mbox{such that} \; \; \; C \supseteq E, \; D \subseteq F, \; \;
\hat{\mu}(C,D) \geq \hat{\mu}(E,F).
\]

\noindent {\bf{Boundary conditions}}
\begin{enumerate}
\item $\mu^{+}({\cal{J}}, \emptyset) = 1$, i.e., ${\displaystyle \sum_{j \in {\cal{J}}}a_{j}^{+} +
\sum_{\{j, k \} \subseteq {\cal{J}}}a_{jk}^{+} = 1}$
\item $\mu^{-}(\emptyset, {\cal{J}}) = 1$, i.e., ${\displaystyle \sum_{j \in {\cal{J}}}a_{j}^{-} +
\sum_{ \{j, k \} \subseteq {\cal{J}}}a_{jk}^{-} = 1}$
\end{enumerate}

\vspace{0.3cm}

\subsection{The $2$-additive bipolar Choquet integral}

\noindent The following theorem gives an expression of $ Ch^{B+}(x, {\mu}^{+})$ and $Ch^{B-}(x, {\mu}^{-})$ considering a 2-additive decomposable bicapacity $\mu$.

\vspace{0.3cm}

\begin{theorem}\label{Bi-polar_Choq} 
{\it Given a 2-additive decomposable bicapacity $\hat{\mu}$, then for all $x \in {\rio }^n$
\begin{enumerate}

\item ${\displaystyle Ch^{B+}(x, {\mu}^{+}) = \sum_{j \in {\cal{J}}, x_j> 0}a^{+}_{j}x_j + \sum_{j, k \in {\cal{J}}, j \neq k, x_j, x_k > 0}a^{+}_{jk}\min\{x_j,
 x_k\}+ \sum_{j, k \in {\cal{J}}, j\neq k, x_j > 0, x_k < 0}a^{+}_{j \vert k}\min\{x_j, - x_k\} }
$

\item ${\displaystyle Ch^{B-}(x, {\mu}^{-}) = -\sum_{j \in {\cal{J}}, x_j< 0}a^{-}_{j}x_j - \sum_{j, k \in {\cal{J}}, j \neq k, x_j, x_k < 0}a^{-}_{jk}\max\{x_j,
 x_k\}- \sum_{j, k \in {\cal{J}}, j\neq k, x_j > 0, x_k < 0}a^{-}_{j \vert k}\max\{-x_j, x_k\} }$
\end{enumerate}
}
\end{theorem}
\begin{proof} 
We shall prove only part 1. Proof of part 2. can be obtained analogously.\\
If the bicapacity $\hat\mu$ is
$2-$additive decomposable, then

\[
\begin{array}{ll}
{\displaystyle Ch^{B+}(x, {\mu}^{+})} &  {\displaystyle = \sum_{j \in
{\cal{J}}^>}\vert x_{(j)}\vert \big[{\mu}^{+}(C_{(j)}, D_{(j)}) -
{\mu}^{+}(C_{(j+1)}, D_{(j+1)}) \Big] =} \\ [6mm]
 & \\
 & {\displaystyle = \sum_{j \in {\cal{J}}^>}\vert x_{(j)} \vert \Big[\Big( \sum_{k \in {\cal{J}}^>, x_k
 \geq \vert x_{(j)} \vert} a^{+}_{k} - \sum_{k \in {\cal{J}}^>, x_k \geq
 \vert x_{(j+1)} \vert} a^{+}_{k}\Big) + } \\ [6mm]
 & {\displaystyle + \Big(\sum_{h, k \in {\cal{J}}^>, h \neq k, x_h, x_k \geq \vert x_{(j)} \vert} a^{+}_{hk} -
  \sum_{h, k \in {\cal{J}}^>, h \neq k, x_h, x_k \geq \vert x_{(j+1)} \vert} a^{+}_{hk}\Big) + } \\ [6mm]
 & {\displaystyle +\Big(\sum_{h, k \in {\cal{J}}^>, h\neq k, x_h, -x_k \geq \vert x_{(j)} \vert} a^{+}_{h \vert k} -
 \sum_{h, k \in {\cal{J}}^>, h\neq k, x_h, -x_k \geq \vert x_{(j+1)} \vert} a^{-}_{h \vert k} \Big)\Big] } \\ [6mm]
 \end{array}
\]

Let us remark that,

\[
a) \; \; \; \; \; \Big( \sum_{k \in {\cal{J}}^>, x_k
 \geq \vert x_{(j)} \vert} a^{+}_{k} - \sum_{k \in {\cal{J}}^>, x_k
 \geq \vert x_{(j+1)} \vert} a^{+}_{k}\Big) =
 \left\{
 \begin{array}{lll}
 \displaystyle\sum_{k\in{\cal J}^{>}, x_{k}=|x_{(j)}|} a^{+}_{k} & \mbox{if} & |x_{(j)}| < |x_{(j+1)}| \\[10mm]
 0 &  & \mbox{otherwise}
 \end{array}
 \right.
\]

\[
b) \; \; \; \; \; \Big( \sum_{k \in {\cal{J}}^>, -x_k
 \geq \vert x_{(j)} \vert} a^{-}_{k} - \sum_{k \in {\cal{J}}^>, -x_k
 \geq \vert x_{(j+1)} \vert} a^{-}_{k}\Big) =
 \left\{
 \begin{array}{lll}
 \displaystyle\sum_{k\in{\cal J}^{>}, -x_{k}=|x_{(j)}|} a^{-}_{k} & \mbox{if} & |x_{(j)}| < |x_{(j+1)}| \\[10mm]
 0 &  & \mbox{otherwise}
 \end{array}
 \right.
\]

\[
c) \; \; \; \; \; \Big( \sum_{\substack{h,k \in {\cal{J}}^>, h\neq k, \\x_h,x_k \geq \vert x_{(j)} \vert }} a^{+}_{hk} - \sum_{\substack{h,k \in {\cal{J}}^>, h\neq k, \\x_h,x_k \geq \vert x_{(j+1)} \vert }}
a^{+}_{hk} \Big) =
 \left\{
 \begin{array}{lll}
{\displaystyle \sum_{\substack{h,k \in {\cal{J}}^>, h\neq k, \\ \min\{x_h,x_k\} = \vert x_{(j)} \vert }}} a^{+}_{hk} & \mbox{if} & |x_{(j)}| < |x_{(j+1)}| \\[10mm]
 0 &  & \mbox{otherwise}
 \end{array}
 \right.
\]

\noindent Considering $a)-c)$ we get that:

\[
\begin{array}{ll}
\chi) & {\displaystyle = \sum_{\substack{j \in {\cal{J}}^>,\\|x_{(j)}|<|x_{(j+1)}|}}\vert x_{(j)}
\vert \Big[  \displaystyle\sum_{k\in{\cal J}^{>}, x_{k}=|x_{(j)}|} a^{+}_{k} + \sum_{\substack{h,k \in {\cal{J}}^>, h\neq k, \\ \min\{x_h,x_k\} = \vert x_{(j)} \vert }} a^{+}_{hk} + \sum_{\substack{h,k \in {\cal{J}}^>, h\neq k, \\\min\{x_h,-x_k\} = \vert x_{(j)} \vert }} a^{+}_{h \vert k} \Big]}\\ [6mm]
\end{array}
\]
\noindent and from this it follows the thesis.
\end{proof}

In the following, we provide the symmetry conditions of Propositions \ref{Sym_1} and \ref{Salvo} in function of the parameters $a_j^{+}$, $a_j^{-}$, $a_{jk}^{+}$, $a_{jk}^{-}$, $a_{j\vert k}^{+}$ and $a_{j\vert k}^{-}$.

\begin{prop}\label{Prop_cond}
{\it Given a 2-additive decomposable bicapacity $\hat\mu$, then $\hat{\mu}(C,D)= - \hat{\mu}(D,C)$ for each
$(C,D) \in P({\cal{J}})$ iff
 \begin{enumerate}
 \item for each $j \in {\cal{J}}$, $a^{+}_{j} = a^{-}_{j}$,
 \item for each $\{j,k\} \subseteq {\cal{J}}$, $a^{+}_{jk} = a^{-}_{jk}$,
 \item for each $ j, k \in {\cal{J}}$, $j \neq k$, $a^{+}_{j \vert k} - a^{-}_{j \vert k} = a^{-}_{k \vert j} - a^{+}_{k \vert j}$.
 \end{enumerate}
 }
\end{prop}
\begin{proof}
First, let us prove that

\begin{itemize}
\item[$(a)$] $\hat{\mu}(C,D) = - \hat{\mu}(D,C)$
\end{itemize}

\noindent implies $1.$, $2.$ and $3.$ For each $j \in \cal{J}$,

\begin{itemize}
\item[$(b)$] $\hat{\mu}(\{j\}, \emptyset)= a^{+}_{j}$  and  $\hat{\mu}(\emptyset,\{j\})= -a^{-}_{j}$
\end{itemize}

\noindent From $(a)$ and $(b)$ we have,

\[
a^{+}_{j}= \hat{\mu}(\{j\}, \emptyset )= - \hat{\mu}(\emptyset
,\{j\})=a^{-}_{j}
\]

\noindent which is $1$.

For each $\{ j,k \} \subseteq \cal{J}$ we have that,
\begin{itemize}
\item[$(c)$]  $\hat{\mu}(\{j,k\}, \emptyset) = a^{+}_{j} + a^{+}_{k} + a^{+}_{jk}$ and
$\hat{\mu}(\emptyset, \{j,k\}) = -a^{-}_{j} - a^{-}_{k} -
a^{-}_{jk}$
\end{itemize}

Being $\hat{\mu}(\{j,k\}, \emptyset)=-\hat{\mu}(\emptyset,\{j,k\})$, and being $a^{+}_{j} = a^{-}_{j}$ and
$a^{+}_{k} = a^{-}_{k}$ by $1.$, we have that for each $\{j,k\} \subseteq \cal{J}$, $a^{+}_{jk}= a^{-}_{jk}$, i.e. $2$.

For all $j,k\in {\cal J}$ with $j\neq k,$ we have:\\
$$\hat{\mu}(\{j\},\{k\})=a_{j}^{+}-a_{k}^{-}+a^{+}_{j \vert k}-a^{-}_{j \vert k}$$
$$\hat{\mu}(\{k\},\{j\})=a_{k}^{+}-a_{j}^{-}+a^{+}_{k \vert j}-a^{-}_{k \vert j}$$
\noindent Being $\hat{\mu}(\{j\},\{k\})=-\hat{\mu}(\{k\},\{j\})$ and having proved that $a^{+}_{j}=a^{-}_{j}, \forall j$, we obtain that 
$a^{+}_{j \vert k}-a^{-}_{j \vert k}=-a^{+}_{k \vert j}+a^{-}_{k \vert j}$ i.e. 3.

\vspace{0.3cm}

It is straightforward to prove that $1.$, $2.$, and $3.$ imply $\hat\mu(C,D)=-\hat\mu(D,C)$.

\end{proof}

\begin{corollary}\label{lem_SYM}
Given a 2-additive decomposable bicapacity $\hat\mu$, $Ch^{B}(P^{B}(a,b),\hat\mu)=-Ch^{B}(P^{B}(b,a),\hat\mu)$ for all $a,b\in A$ iff 
\begin{enumerate}
\item for each $j \in {\cal{J}}$, $a^{+}_{j} = a^{-}_{j}$,
\item for each $\{j,k\} \subseteq {\cal{J}}$, $a^{+}_{jk} = a^{-}_{jk}$,
\item for each $ j, k \in {\cal{J}}$, $j \neq k$, $a^{+}_{j \vert k} - a^{-}_{j \vert k} = a^{-}_{k \vert j} - a^{+}_{k \vert j}$. 
\end{enumerate}
\end{corollary}
\begin{proof}
It follows by Propositions \ref{Prop_cond} and \ref{Sym_1}.
\end{proof}

\begin{prop}\label{Prop_cond_2}
{\it Given a 2-additive decomposable bicapacity $\hat\mu$, then ${\mu}^{+}(C,D)= {\mu}^{-}(D,C)$ for each
$(C,D) \in P({\cal{J}})$ iff
 \begin{enumerate}
 \item for each $j \in {\cal{J}}$, $a^{+}_{j} = a^{-}_{j}$,
 \item for each $\{j,k\} \subseteq {\cal{J}}$, $a^{+}_{jk} = a^{-}_{jk}$,
 \item for each $ j, k \in {\cal{J}}$, $j \neq k$, $a^{+}_{j \vert k} = a^{-}_{k \vert j}$.
 \end{enumerate}
 }
\end{prop}
\begin{proof}
Analogous to Proposition \ref{Prop_cond}.
\end{proof}

\begin{corollary}\label{lem_SYM_2}
Given a 2-additive decomposable bicapacity $\hat\mu$, $Ch^{B+}(P^{B}(a,b),\mu^{+})=Ch^{B-}(P^{B}(b,a),\mu^{-})$ for all $a,b\in A$ iff 
\begin{enumerate}
\item for each $j \in {\cal{J}}$, $a^{+}_{j} = a^{-}_{j}$,
\item for each $\{j,k\} \subseteq {\cal{J}}$, $a^{+}_{jk} = a^{-}_{jk}$,
\item for each $ j, k \in {\cal{J}}$, $j \neq k$, $a^{+}_{j \vert k}= a^{-}_{k \vert j}$. 
\end{enumerate}
\end{corollary}
\begin{proof}
It follows by Propositions \ref{Prop_cond_2} and \ref{Salvo}.
\end{proof}

\noindent Because the first two conditions of Proposition \ref{Sym_1} are the same of the first two conditions of Proposition \ref{Salvo}, but the third condition of Proposition \ref{Salvo} implies the third one of Proposition \ref{Sym_1}, in order to get both $Ch^{B}(P^{B}(a,b),\hat\mu)=-Ch^{B}(P^{B}(b,a),\hat\mu)$ and $Ch^{B+}(P^{B}(a,b),\mu^{+})=Ch^{B-}(P^{B}(b,a),\mu^{-})$ for all $a,b\in A$, we impose that shoul be fulfilled the conditions in Proposition \ref{Salvo}.

\section{Assessing the preference information}\label{assessing}
On the basis of the considered $2$-additive decomposable bicapacity $\hat\mu$, and holding the symmetry condition in Corollary \ref{lem_SYM_2}, we propose the
following methodology which simplifies the assessment of the
preference information. \\We consider the following information
provided by the DM and their representation in terms of linear
constraints:

\begin{enumerate}
\item {\it Comparing pairs of actions locally or globally}. The constraints represent some pairwise comparisons on a set of
training actions. Given two actions $a$ and $b$, the DM may
prefer $a$ to $b$, $b$ to $a$ or be indifferent to both:
    \begin{enumerate}
    \item the linear constraint associated with $a{\cal P}b$ ($a$ is locally preferred to $b$) is: 
    $$Ch^{B}(P^{B}(a, b), \hat{\mu}) > 0;$$
    \item the linear constraints associated with $a{\cal P}^{I}b$ ($a$ is preferred to $b$ with respect to the bipolar PROMETHEE I method) are:     
    $$    
    \left.
    \begin{array}{l}
    \Phi^{B+}(a)\geq\Phi^{B+}(b),\\
    \Phi^{B-}(a)\leq\Phi^{B-}(b),\\
    \Phi^{B+}(a)-\Phi^{B-}(a)>\Phi^{B+}(b)-\Phi^{B-}(b),\\
    \end{array}
    \right\} 
    $$
    \item the linear constraint associated with $a{\cal P}^{II}b$ ($a$ is preferred to $b$ with respect to the bipolar PROMETHEE II method) is: 
    $$    
    \Phi^{B}(a)>\Phi^{B}(b)
    $$    
    \item the linear constraint associated with $a{\cal I}b$ ($a$ is locally indifferent to $b$) is:
    $$
    Ch^{B}(P^{B}(a, b), \hat{\mu}) = 0
    $$
    \item the linear constraints associated with $a{\cal I}^{I}b$ ($a$ is indifferent to $b$ with respect to the bipolar PROMETHEE I method) are:    
    $$    
    \left.
    \begin{array}{l}
    \Phi^{B+}(a)=\Phi^{B+}(b),\\
    \Phi^{B-}(a)=\Phi^{B-}(b),\\
    \end{array}
    \right\}
    $$
    \item the linear constraint associated with $a{\cal I}^{II}b$ ($a$ is indifferent to $b$ with respect to the bipolar PROMETHEE II method) is:
    $$
    \Phi^{B}(a)=\Phi^{B}(b)
    $$
    \end{enumerate}
\item {\it Comparison of the intensity of preferences between pairs of
actions}. The constraints represent some pairwise comparisons between pairs of alternatives on a set of
training actions. Given four actions $a$, $b$, $c$ and $d$:
\begin{enumerate}
\item the linear constraints associated with $(a,b) {\cal{P}} (c, d)$ (the
local preference of $a$ over $b$ is larger than the
local preference of $c$ over $d$) is:
\[
Ch^{B}(P^{B}(a, b), \hat{\mu}) > Ch^{B}(P^{B}(c, d), \hat{\mu})
\]
\item the linear constraints associated with $(a,b) {\cal{I}} (c, d)$ (the
local preference of $a$ over $b$ is the same of
local preference of $c$ over $d$) is:
\[
Ch^{B}(P^{B}(a, b), \hat{\mu}) = Ch^{B}(P^{B}(c, d), \hat{\mu}) 
\]
\end{enumerate}
\item {\it Importance of criteria}.
A partial ranking over the set of criteria $\cal{J}$ may be
provided by the DM: 
    \begin{enumerate}
    \item criterion $g_j$ is more important than criterion $g_k$, which leads
    to the constraint $a_{j} > a_{k}$;
    \item criterion $g_j$ is equally important to criterion $g_k$, which leads
    to the constraint $a_{j} = a_{k}$.
    \end{enumerate}
\item {\it The sign of interactions}. The DM may be able, for certain cases,
to provide the sign of some interactions. For example, if there is
a synergy effect when criterion $g_j$ interacts with criterion
$g_k$, the following constraint
should be added to the model: $a_{jk} > 0$. 
\item {\it Interaction between pairs of criteria}. The DM can provide
some information about interaction between criteria:
     \begin{enumerate}
    \item[a)] if the DM feels that interaction between $g_j$ and $g_k$ is greater
    than the interaction between $g_p$ and $g_q$, the constraint should be
    defined as follows: $|a_{jk}| > |a_{pq}|$ where in particular:
    \begin{itemize}
    \item if both couples of criteria are synergic then: $a_{jk} > a_{pq}$,
    \item if both couples of criteria are redundant then: $a_{jk} < a_{pq}$,
    \item if $(j,k)$ is a couple of synergic criteria and $(p,q)$ is a couple of redundant criteria, then: $a_{jk} > -a_{pq}$,
    \item if $(j,k)$ is a couple of redundant criteria and $(p,q)$ is a couple of synergic criteria, then: $-a_{jk} > a_{pq}$.
    \end{itemize}
    \item[b)] if the DM feels that the strength of the interaction between $g_j$ and $g_k$ is the same of the strength of the interaction between
    $g_p$ and $g_q$, the constraint will be the following: $|a_{jk}| = |a_{pq}|$ and in particular:
    \begin{itemize}
    \item if both couples of criteria are synergic or redundant then: $a_{jk} = a_{pq}$,
    \item if one couple of criteria is synergic and the other is redundant then: $a_{jk} = -a_{pq}$,
    \end{itemize}     
    \end{enumerate}
\item {\it The power of the opposing criteria}{\label{interact}}. Concerning the power
of the opposing criteria several situations may occur. For
example:
    \begin{enumerate}
    \item[a)] when the opposing power of $g_k$ is larger than the
    opposing power of $g_h$, with respect to $g_j$, which expresses a positive preference,
    we can define the following constraint:
    $a^{+}_{j \vert k} < a^{+}_{j \vert h}$ (because $a^{+}_{j \vert h}\leq 0$ and $a^{-}_{j \vert h}\leq 0$ for all $j,k$ with $j\neq k$);
    \item[b)] if the opposing power of $g_k$, expressing negative preferences, is larger with $g_j$ rather
    than with $g_h$, the constraint will be $a^{+}_{j \vert k} < a^{+}_{h \vert
    k}$.
    \end{enumerate}
\end{enumerate}

\subsection{A linear programming model}
All the constraints presented in the previous section
along with the symmetry, boundary and monotonicity conditions can now be put together and form a system of
linear constraints. Strict inequalities can be converted into weak inequalities by adding a variable $\varepsilon$. It is well-know that
such a system has a feasible solution if and only if when
maximizing $\varepsilon$, its value is strictly positive
\cite{MarichalRo00}. Considering constraints given by Corollary \ref{lem_SYM_2} for the symmetry condition, the linear programming model can be stated
as follows (where $j{\cal{P}}k$ means that criterion $g_j$ is
more important than criterion $g_k$; the remaining relations have a similar interpretation):

\begin{scriptsize}
    \begin{displaymath}
    \begin{array}{l}
    \mbox{Max}\; \varepsilon \\[5mm]
      \left.
      \begin{array}{ll}
           
             Ch^{B}(P^{B}(a, b),\hat\mu) \geq \varepsilon \;\; \mbox{if} \;\; a{\cal{P}}b, & Ch^{B}(P^{B}(a, b),\hat\mu) = 0  \;\; \mbox{if} \;\; a{\cal{I}}b, \\[1mm]          
             \left.
             \begin{array}{l}
             \Phi^{B+}(a)\geq\Phi^{B+}(b),\\
             \Phi^{B-}(a)\leq\Phi^{B-}(b),\\
             \Phi^{B+}(a)-\Phi^{B-}(a)\geq \Phi^{B+}(b)-\Phi^{B-}(b)+\varepsilon\\
             \end{array}
             \right\} \;\; \mbox{if} \;\; a{\cal P}_B^{I}b & 
             \left.
             \begin{array}{l}
             \Phi^{B+}(a)=\Phi^{B+}(b),\\
             \Phi^{B-}(a)=\Phi^{B-}(b)\\           
             \end{array}
             \right\} \;\; \mbox{if} \;\; a{\cal I}_B^{I}b \\[1mm]
           \Phi^{B}(a)\geq\Phi^{B}(b)+\varepsilon \;\; \mbox{if} \;\; a{\cal P}_B^{II}b & \Phi^{B}(a)=\Phi^{B}(b) \;\; \mbox{if} \;\; a{\cal I}_B^{II}b \\[1mm]
             {Ch^{B}(P^{B}(a, b),\hat\mu) \geq Ch^{B}(P^{B}(c, d),\hat\mu) + \varepsilon } \;\; \mbox{if} \;\; (a,b) {\cal{P}} (c,d), & {Ch^{B}(P^{B}(a, b),\hat\mu) = Ch^{B}(P^{B}(c, d),\hat\mu)}  \;\; \mbox{if} \;\;(a,b){\cal{I}}(c,d), \\[1mm]
            a_{j} - a_{k} \geq \varepsilon \;\; \mbox{if} \;\; j{\cal{P}}k,  & a_{j} = a_{k} \;\; \mbox{if} \;\; j{\cal{I}}k,   \\[1mm] 
            |a_{jk}| - |a_{pq}| \geq \varepsilon \;\; \mbox{if} \;\;  \{ j, k\} {\cal{P}}\{ p, q\}, \; \mbox{(see point 5.a) of the previous subsection )}& \\[1mm]
            |a_{jk}| = |a_{pq}|                \;\;\mbox{if} \;\;  \{ j, k\} {\cal{I}}\{ p, q\}, \; \mbox{(see point 5.b) of the previous subsection )} &  \\[1mm]  
            a_{jk} \geq \varepsilon  \;\; \mbox{if there is synergy between criteria $j$ and $k$}, \\[1mm]
            a_{jk} \leq - \varepsilon \;\; \mbox{if there is redundancy between criteria $j$ and $k$}, \\[1mm]
            a_{jk} = 0  \;\;\mbox{if criteria $j$ and $k$ are not interacting}, \\[1mm]
           \mbox{Power of the opposing criteria of the type \ref{interact}:}\\[1mm]
           a^{+}_{j \vert k} - a^{+}_{j \vert p} \geq \varepsilon, & a^{-}_{j \vert k} - a^{-}_{j \vert p} \geq \varepsilon,  \\[1mm]
           a^{+}_{j \vert k} - a^{+}_{p \vert k} \geq \varepsilon, & a^{-}_{j \vert k} - a^{-}_{p \vert k} \geq \varepsilon,  \\[1mm]           
           \mbox{Symmetry conditions (Proposition \ref{lem_SYM_2}):}\\[1mm] 
           {\displaystyle  a_{j \vert k}^{+} = a_{k \vert j}^{-}, \; \; \forall \; j, k\in {\cal{J}}, j\neq k \;\; \; }\\[1mm]
           \mbox{Boundary and monotonicity conditions:}\\[1mm] 
           {\displaystyle \sum_{j \in {\cal{J}}}a_{j} + \sum_{\{j, k \} \subseteq {\cal{J}}}a_{jk} = 1},\\[1mm]
           {\displaystyle a_{j}\geq 0 \; \; \forall \; j \in {\cal{J}}}, & {\displaystyle a_{j \vert k}^{+}, \; a_{j \vert k}^{-} \; \leq 0 \; \; \forall \; j, k \in {\cal{J}}},\\[1mm]
           {\displaystyle a_{j} + \sum_{k \in C}a_{jk} + \sum_{k \in D}a^{+}_{j \vert k} \geq 0, \; \; \forall \;
           j \in {\cal{J}}, \; \forall  (C \cup \{ j \}, D) \in P({\cal{J}}) }, \\[1mm]
            {\displaystyle a_{j} + \sum_{k \in D}a_{jk} + \sum_{h \in C}a^{-}_{h \vert j} \geq 0, \; \; \forall \;
           j \in {\cal{J}}, \; \forall  (C, D \cup \{ j \}) \in P({\cal{J}}) }. \\           
    \end{array}
    \right\}E^{A^R}
    \end{array}
    \end{displaymath}
\end{scriptsize}




\subsection{Restoring PROMETHEE}
The condition which allows to restore the classical PROMETHEE methods is the following:
\begin{enumerate}
\item $\forall j, k \in {\cal{J}}, \; \; a_{jk} = a^{+}_{j \vert k} = a^{-}_{j \vert k} = 0$.
\end{enumerate}

\noindent If Condition 1. is not satisfied and the following condition holds
\begin{enumerate}
	\item[2.] $\forall j, k \in {\cal{J}}, a^{+}_{j \vert k} = a^{-}_{j \vert k} = 0$,
\end{enumerate}
\noindent then the comprehensive preference of $a$ over $b$ is calculated as
the difference between the Choquet integral of the positive
preferences and the Choquet integral of the negative preferences,
with a common capacity $\mu$ on ${\cal J}$ for the positive and the negative
preferences, i.e. there exists $\mu: 2^{\cal J}\rightarrow[0,1]$, with $\mu(\emptyset)=0,$ $\mu({\cal J})=1,$ and $\mu(A)\leq\mu(B)$ for all $A\subseteq B\subseteq{\cal J}$, such that 
$$Ch^{B}(P^{B}(a, b), \hat{\mu})=\int_0^1\mu(\{j\in{\cal J}:P_{j}^{B}(a,b)>t\})dt-\int_0^1 \mu(\{j\in{\cal J}:P_{j}^{B}(a,b)<-t\})dt.$$
We shall call this type of aggregation of
preferences, the symmetric Choquet integral PROMETHEE method.\\
If neither 1. nor 2. are satisfied, but the following condition holds
\begin{enumerate}
	\item[3.] $\forall j, k \in {\cal{J}}, a^{+}_{j \vert k} = a^{-}_{k \vert j}$,
\end{enumerate}
\noindent then we have the Bipolar PROMETHEE methods.\\

\subsection{A constructive learning preference information elicitation process}\label{constructive process}
The previous Conditions $1.$-$3.$ suggest a proper
way to deal with the linear programming model in order to assess
the interactive bipolar criteria coefficients. Indeed, it is very
wise trying before to elicit weights concordant with the classical
PROMETHEE method. If this is not possible, one can consider a
PROMETHEE method which aggregates positive and negative
preferences using the Choquet integral. If this is not possible, one can consider 
the bipolar symmetric PROMETHEE method. If, by proceeding in this
way, we are not able to represent the DM's preferences, then we can
take into account a more sophisticated aggregation procedure by
using the bipolar PROMETHEE method. This way to progress from the
simplest to the most sophisticated model can be outlined in a
four steps procedure as follows:

\begin{enumerate}
\item Solve the linear programming problem

\begin{equation}\label{First_PROB}
\begin{array}{l}
\;\;\mbox{Max}\; \varepsilon=\varepsilon_{1} \\[2mm]
\left.
\begin{array}{ll}
 E^{A^{R}}\\
 a_{jk} = a^{+}_{j \vert k} = a^{-}_{j \vert k} = 0,\;\;\forall j, k \in {\cal{J}}
\end{array}
\right\}E_{1}
\end{array}
\end{equation}

adding to $E^{A^{R}}$ the constraint
related to the previous Condition $1.$ If $E_{1}$ is feasible and $\varepsilon_{1} > 0$, then the obtained preferential
parameters are concordant with the classical PROMETHEE method.
Otherwise,
\item Solve the linear programming problem

\begin{equation}\label{Second_PROB}
\begin{array}{l}
\;\;\mbox{Max}\; \varepsilon=\varepsilon_{2} \\[2mm]
\left.
\begin{array}{ll}
 E^{A^{R}}\\
 a^{+}_{j \vert k} = a^{-}_{j \vert k} = 0, \;\;\forall j, k \in {\cal{J}}
\end{array}
\right\}E_{2}
\end{array}
\end{equation}

adding to $E^{A^{R}}$ the constraint related to the previous Condition $2$. If $E_{2}$ is feasible and $\varepsilon_2 > 0$, then the information
is concordant with the symmetric Choquet integral PROMETHEE method having a unique capacity for the negative and the positive part.
Otherwise,

\item Solve the linear programming problem

\begin{equation}\label{third_PROB}
\begin{array}{l}
\;\;\mbox{Max}\; \varepsilon=\varepsilon_{3} \\[2mm]
\;\;E^{A^{R}}\\
\end{array}
\end{equation}

\noindent If $E_{3}$ is feasible and $\varepsilon_3 > 0$, then the information
is concordant with the bipolar PROMETHEE method. Otherwise,
\item We can try to help the DM by providing some information
about inconsistent judgments, when it is the case, by using a
similar constructive learning procedure proposed in
\cite{MousseauFiDiGoCl03}. In fact, in the linear programming model some of the constraints
cannot be relaxed, that is, the basic properties of the model
(symmetry, boundary and monotonicity conditions). The remaining constraints can lead
to an infeasible linear system which means that the DM provided
inconsistent information about her/his preferences. The methods
proposed in \cite{MousseauFiDiGoCl03} can then be used in this
context, providing to the DM some useful information about
inconsistent judgments.
\end{enumerate}

\section{ROR and Bipolar PROMETHEE methods}
In the above sections we dealt with the problem of finding a bicapacity restoring preference information provided by the DM in case where multiple criteria evaluations are aggregated by Bipolar PROMETHEE method. Generally, there could exist more than one model (in our case the model will be a bicapacity, but in other contexts it could be a utility function or an outranking relation) compatible with the preference information provided by the DM on the training set of alternatives. Each compatible model restores the preference information provided by the DM but two different compatible models could compare the other alternatives not provided as examples by the DM in a different way. For this reason, the choice of one of these models among those compatible could be considered arbitrary. In order to take into account not only one but the whole set of models compatible with the preference information provided by the DM, we consider the ROR \cite{greco2010robust}. This approach considers the whole set of models compatible with preference information provided by the DM building two preference relations: the weak \textit{necessary} preference relation, for which alternative $a$ is necessarily weakly preferred to alternative $b$ (and we write $a\succsim^{N}b$), if $a$ is at least as good as $b$ for all compatible models, and the weak \textit{possible} preference relation, for which alternative $a$ is possibly weakly preferred to alternative $b$ (and we write $a\succsim^{P}b$), if $a$ is at least as good as $b$ for at least one compatible model. \\
Considering the bipolar flows (\ref{pos_flow})-(\ref{net_flow}) and the comprehensive Choquet integral in equation (\ref{bipolarpref}), given the alternatives $a,b\in A$, we say that $a$ outranks $b$ (or $a$ is at least as good as $b$):
\begin{itemize}
\item locally, if $Ch^{B}(P^{B}(a, b), \hat{\mu})\geq 0$;
\item globally and considering the bipolar PROMETHEE I method, if $\Phi^{B+}(a)\geq\Phi^{B+}(b)$, $\Phi^{B-}(a)\leq\Phi^{B-}(b)$;
\item globally and considering the bipolar PROMETHEE II method, if $\Phi^{B}(a)\geq\Phi^{B}(b).$
\end{itemize}

\noindent To check if $a$ is necessarily preferred to $b$, we look if it is possible that $a$ does not outrank $b$. Locally, this means that it is possible that there exists a bicapacity $\hat\mu$ such that $Ch^{B}(P^{B}(a,b),\hat\mu)<0$; globally, considering the bipolar PROMETHEE I this means that $\Phi^{B+}(a)<\Phi^{B+}(b)$ or $\Phi^{B-}(a)>\Phi^{B-}(b)$, while considering the bipolar PROMETHEE II this means that $\Phi^B(a)<\Phi^B(b)$. Given the following set of constraints, 

$$
\label{SN}
\left.
\begin{array}{l}
  E^{A^R}\\[2mm]                
          
  \mbox{if one verifies the truth of global outranking: }\\[2mm]  
  \mbox{\quad if exploited in the way of the bipolar PROMETHEE II method, then: }\\[2mm]  \quad \quad \Phi^B(a) + \varepsilon \leq \Phi^B(b)\\[2mm]     
  \mbox{\quad if exploited in the way of the bipolar PROMETHEE I method, then: }\\[2mm]             
  \mbox {\quad \quad }  \Phi^{B+}(a) + \varepsilon \leq \Phi^{B+}(b) + 2M_1 \mbox{ \ and \ }  \Phi^{B-}(a) + 2M_2 \geq \Phi^{B-}(b) + \varepsilon \\[2mm]
  \mbox { \quad \quad }  \mbox{ where } M_i \in \{0,1\}, i=1,2, \mbox{ and } \sum_{i=1}^2 M_i \leq 1 \\[2mm]
          \mbox{if one verifies the truth of local outranking: }\\[2mm]  
           \mbox { \quad \quad } Ch^{B}(P^{B}(a,b),\hat\mu)+\varepsilon \leq 0                      
\end{array}
\right\} E^{N}(a,b)
$$
\noindent we say that $a$ is weakly necessarily preferred to $b$ if $E^{N}(a,b)$ is infeasible or $\varepsilon^{*}\leq 0$ where $\varepsilon^{*}=\max\varepsilon$ s.t. $E^{N}(a,b)$. 

\noindent To check if $a$ is possibly preferred to $b$, we check if it is possible that $a$ outrank $b$ for at least one bicapacity $\hat\mu.$ Locally, this means that there exists a bicapacity $\hat\mu$ such that $Ch^{B}(P^{B}(a,b),\hat\mu)\geq 0$; globally, considering PROMETHEE I this means that $\Phi^{B+}(a)\geq\Phi^{B+}(b)$ and $\Phi^{B-}(a)\leq\Phi^{B-}(b)$, while considering PROMETHEE II this means that $\Phi^{B}(a)\geq\Phi^{B}(b)$. Given the following set of constraints, 

$$
\label{SP}
\left.
\begin{array}{l}
  E^{A^R}\\[2mm]                
          
  \mbox{if one verifies the truth of global outranking: }\\[2mm]  
  \mbox{\quad if exploited in the way of the bipolar PROMETHEE II method, then: }\\[2mm]  \quad \quad \Phi^B(a) \geq \Phi^B(b)\\[2mm]     
  \mbox{\quad if exploited in the way of the bipolar PROMETHEE I method, then: }\\[2mm]             
  \mbox {\quad \quad }  \Phi^{B+}(a) \geq \Phi^{B+}(b) \mbox{ \ and \ }  \Phi^{B-}(a) \leq \Phi^{B-}(b) \\[2mm]
          \mbox{if one verifies the truth of local outranking: }\\[2mm]  
           \mbox { \quad \quad } Ch^{B}(P^{B}(a,b),\hat\mu) \geq 0                      
\end{array}
\right\} E^{P}(a,b)
$$
\noindent we say that $a$ is weakly possibly preferred to $b$ if $E^{P}(a,b)$ is feasible and $\varepsilon^{*}> 0$ where $\varepsilon^{*}=\max\varepsilon$ s.t. $E^{P}(a,b)$.

\section{Didactic Example}
Inspired by an example in literature \cite{Grabisch1996}, let us consider the problem of evaluating High School students according to their grades in Mathematics, Physics and Literature. In the following we suppose that the Director is the DM, while we will cover the role of analyst helping and supporting the DM in (her)his evaluations. \\
The Director thinks that scientific subjects (Mathematics and Physics) are more important than Literature. However, when students $a$ and $b$ are compared, if $a$ is better than $b$ both at Mathematics and Physics but $a$ is much worse than $b$ at Literature, then the Director has some doubts about the comprehensive preference of $a$ over $b$.\\
Mathematics and Physics are in some sense \textit{redundant} with respect to the comparison of students, since usually students which are good at Mathematics are also good at Physics. As a consequence, if $a$ is better than $b$ at Mathematics, the comprehensive preference of the student $a$ over the student $b$ is stronger if $a$ is better than $b$ at Literature rather than if $a$ is better than $b$ at Physics. \\
Let us consider the students whose grades (belonging to the range $\left[0,20\right]$) are represented in Table \ref{Evaluations} and the following formulation of the preference of $a$ over $b$ with respect to each criterion $g_j$, for all $j=(M)$ Mathematics, $(Ph)$ Physics, $(L)$ Literature.

\begin{table}[htbp]
\begin{center}
\begin{tabular}{cccc}
\hline
\mbox{Students} & \mbox{Mathematics} & \mbox{Physics} & \mbox{Literature}\\
\hline
$s_1$   & 16 & 16 & 16 \\
$s_2$   & 15 & 13 & 18 \\
$s_3$   & 19 & 18 & 14 \\
$s_4$   & 18 & 16 & 15 \\
$s_5$   & 15 & 16 & 17 \\
$s_6$   & 13 & 13 & 19 \\
$s_7$   & 17 & 19 & 15 \\
$s_8$   & 15 & 17 & 16 \\
\hline
\end{tabular}
\end{center}
\caption{Evaluations of the students}\label{Evaluations}
\end{table}

$$
P_j(a,b)=
\left\{
\begin{array}{ccc}
0                  & \mbox{if} & g_{j}(b)\geq g_j(a)   \\
(g_j(a)-g_j(b))/4  & \mbox{if} & 0<g_j(a)-g_j(b)\leq 4 \\
1                  &           & \mbox{otherwise}      \\
\end{array}
\right.
$$

From the values of the partial preferences $P_j(a,b)$, we obtain the positive and the negative partial preferences $P_{j}^{B}(a,b)$ with respect to each criterion $g_j$, for $j=M,Ph,L$ using the definition (\ref{equat}).
\noindent Thus, to each pair of students $(s_{i},s_{j})$ is associated a vector of three elements: \\ $P^{B}(s_{i},s_{j})=\left[P^{B}_{M}(s_{i},s_{j}),P^{B}_{Ph}(s_{i},s_{j}),P^{B}_{L}(s_{i},s_{j})\right]$; for example, to the pair of students $(s_{1},s_{2})$ is associated the vector $P^{B}(s_{1},s_{2})=\left[0.25,0.75,-0.5\right]$.


Let us suppose that the Dean provides the following information regarding some pairs of students:
\begin{itemize}
\item student $s_1$ is preferred to student $s_2$ more than student $s_3$ is preferred to student $s_4$,
\item student $s_7$ is preferred to student $s_8$ more than student $s_5$ is preferred to student $s_6$.
\end{itemize}
As explained in section \ref{assessing}, these two information are translated by the constraints: $$Ch^{B}(P^{B}(s_1,s_2),\hat{\mu})>Ch^{B}(P^{B}(s_3,s_4),\hat{\mu}),\;\; \mbox{and} \;\; Ch^{B}(P^{B}(s_7,s_8),\hat{\mu})>Ch^{B}(P^{B}(s_5,s_6),\hat{\mu})$$

\noindent Following the procedure described in section \ref{constructive process}, at first we check if the classical PROMETHEE method and the symmetric Choquet integral PROMETHEE method are able to restore the preference information provided by the Dean; solving the optimization problems \ref{First_PROB} and \ref{Second_PROB}, we get $\varepsilon_{1}<0$ and $\varepsilon_{2}<0$ and therefore neither the classical PROMETHEE method nor the symmetric Choquet integral PROMETHEE method are able to explain the preference information provided by the Dean. Solving the optimization problem \ref{third_PROB}, we get this time $\varepsilon_3>0$; this means that the information provided by the Dean can be explained by the Bipolar PROMETHEE method.\\
In order to better understand the problem at hand, we suggested to the Dean to use the ROR applied to the bipolar PROMETHEE method as discussed in the previous section. Using the first piece of preference information, we get the necessary and possible preference relations shown in Table \ref{first_necessary} at local level and considering PROMETHEE II and PROMETHEE I. In Table \ref{nec_local}, the value 1 in position $(i,j)$ means that $s_{i}$ is necessarily locally preferred to $s_{j}$ while the viceversa corresponds to the value. Analogous meaning have the values 1 and 0 in in Tables \ref{nec_PROM_II} and \ref{nec_PROM_I} respectively.

\begin{table}[!h]
\begin{center}
\caption{Necessary preference relations after the first piece of preference information\label{first_necessary}}
\subtable[Local\label{nec_local}]{%
\resizebox{0.25\textwidth}{!}{
\begin{tabular}{ccccccccc}
\hline
    &  $\mathbf{s_1}$  &  $\mathbf{s_2}$ &  $\mathbf{s_3}$ &  $\mathbf{s_4}$ &  $\mathbf{s_5}$ &  $\mathbf{s_6}$ &  $\mathbf{s_7}$ &  $\mathbf{s_8}$ \\ 
\hline
 $\mathbf{s_1}$  &        0 &          1 &          0 &          0 &          0 &          1 &          0 &          0 \\

 $\mathbf{s_2}$  &        0 &          0 &          0 &          0 &          0 &          0 &          0 &          0 \\

 $\mathbf{s_3}$  &        1 &          1 &          0 &          1 &          0 &          1 &          0 &          0 \\

 $\mathbf{s_4}$  &        0 &          1 &          0 &          0 &          0 &          0 &          0 &          0 \\

 $\mathbf{s_5}$  &        0 &          1 &          0 &          0 &          0 &          1 &          0 &          0 \\

 $\mathbf{s_6}$  &        0 &          0 &          0 &          0 &          0 &          0 &          0 &          0 \\

 $\mathbf{s_7}$  &        1 &          1 &          0 &          0 &          1 &          1 &          0 &          1 \\

 $\mathbf{s_8}$  &        0 &          1 &          0 &          0 &          1 &          1 &          0 &          0 \\

\hline
\end{tabular}
}
}
\subtable[PROMETHEE II\label{nec_PROM_II}]{%
\resizebox{0.25\textwidth}{!}{
\begin{tabular}{ccccccccc}
\hline
    &  $\mathbf{s_1}$  &  $\mathbf{s_2}$ &  $\mathbf{s_3}$ &  $\mathbf{s_4}$ &  $\mathbf{s_5}$ &  $\mathbf{s_6}$ &  $\mathbf{s_7}$ &  $\mathbf{s_8}$ \\ 
\hline
 $\mathbf{s_1}$  &   0 &          0 &          0 &          0 &          0 &          0 &          0 &          0 \\

 $\mathbf{s_2}$  &        0 &          0 &          0 &          0 &          0 &          0 &          0 &          0 \\

 $\mathbf{s_3}$  &        0 &          0 &          0 &          1 &          0 &          0 &          0 &          0 \\

 $\mathbf{s_4}$  &        0 &          0 &          0 &          0 &          0 &          0 &          0 &          0 \\

 $\mathbf{s_5}$  &        0 &          1 &          0 &          0 &          0 &          1 &          0 &          0 \\

 $\mathbf{s_6}$  &        0 &          0 &          0 &          0 &          0 &          0 &          0 &          0 \\

 $\mathbf{s_7}$  &        1 &          1 &          0 &          0 &          1 &          1 &          0 &          1 \\

 $\mathbf{s_8}$  &        0 &          0 &          0 &          0 &          0 &          0 &          0 &          0 \\     
 
\hline
\end{tabular}
}
}
\subtable[PROMETHEE I\label{nec_PROM_I}]{%
\resizebox{0.25\textwidth}{!}{
\begin{tabular}{ccccccccc}
\hline
    &  $\mathbf{s_1}$  &  $\mathbf{s_2}$ &  $\mathbf{s_3}$ &  $\mathbf{s_4}$ &  $\mathbf{s_5}$ &  $\mathbf{s_6}$ &  $\mathbf{s_7}$ &  $\mathbf{s_8}$ \\ 
\hline
 $\mathbf{s_1}$  &        0 &          0 &          0 &          0 &          0 &          0 &          0 &          0 \\

 $\mathbf{s_2}$  &        0 &          0 &          0 &          0 &          0 &          0 &          0 &          0 \\

 $\mathbf{s_3}$  &        0 &          0 &          0 &          0 &          0 &          0 &          0 &          0 \\

 $\mathbf{s_4}$  &        0 &          0 &          0 &          0 &          0 &          0 &          0 &          0 \\

 $\mathbf{s_5}$  &        0 &          0 &          0 &          0 &          0 &          0 &          0 &          0 \\

 $\mathbf{s_6}$  &        0 &          0 &          0 &          0 &          0 &          0 &          0 &          0 \\

 $\mathbf{s_7}$  &        1 &          1 &          0 &          0 &          0 &          0 &          0 &          0 \\

 $\mathbf{s_8}$  &        0 &          0 &          0 &          0 &          0 &          0 &          0 &          0 \\
\hline
\end{tabular}
}
}
\end{center}
\end{table}

\begin{table}[!h]
\begin{center}
\caption{Possible preference relations after the first piece of preference information\label{poss_first}}
\subtable[Local]{%
\resizebox{0.25\textwidth}{!}{
\begin{tabular}{ccccccccc}
\hline
    &  $\mathbf{s_1}$  &  $\mathbf{s_2}$ &  $\mathbf{s_3}$ &  $\mathbf{s_4}$ &  $\mathbf{s_5}$ &  $\mathbf{s_6}$ &  $\mathbf{s_7}$ &  $\mathbf{s_8}$ \\ 
\hline
 $\mathbf{s_1}$  &        0 &          1 &          0 &          1 &          1 &          1 &          0 &          1 \\

 $\mathbf{s_2}$  &        0 &          0 &          0 &          0 &          0 &          1 &          0 &          0 \\

 $\mathbf{s_3}$  &        1 &          1 &          0 &          1 &          1 &          1 &          1 &          1 \\

 $\mathbf{s_4}$  &        1 &          1 &          0 &          0 &          1 &          1 &          1 &          1 \\

 $\mathbf{s_5}$  &        1 &          1 &          1 &          1 &          0 &          1 &          0 &          0 \\

 $\mathbf{s_6}$  &        0 &          1 &          0 &          1 &          0 &          0 &          0 &          0 \\

 $\mathbf{s_7}$  &        1 &          1 &          1 &          1 &          1 &          1 &          0 &          1 \\

 $\mathbf{s_8}$  &        1 &          1 &          1 &          1 &          1 &          1 &          0 &          0 \\    

\hline
\end{tabular}
}
}
\subtable[PROMETHEE II]{%
\resizebox{0.25\textwidth}{!}{
\begin{tabular}{ccccccccc}
\hline
    &  $\mathbf{s_1}$  &  $\mathbf{s_2}$ &  $\mathbf{s_3}$ &  $\mathbf{s_4}$ &  $\mathbf{s_5}$ &  $\mathbf{s_6}$ &  $\mathbf{s_7}$ &  $\mathbf{s_8}$ \\ 
\hline  
 $\mathbf{s_1}$  &        0 &          1 &          1 &          1 &          1 &          1 &          0 &          1 \\

 $\mathbf{s_2}$  &        1 &          0 &          1 &          1 &          0 &          1 &          0 &          1 \\

 $\mathbf{s_3}$  &        1 &          1 &          0 &          1 &          1 &          1 &          1 &          1 \\

 $\mathbf{s_4}$  &        1 &          1 &          0 &          0 &          1 &          1 &          1 &          1 \\

 $\mathbf{s_5}$  &        1 &          1 &          1 &          1 &          0 &          1 &          0 &          1 \\

 $\mathbf{s_6}$  &        1 &          1 &          1 &          1 &          0 &          0 &          0 &          1 \\

 $\mathbf{s_7}$  &        1 &          1 &          1 &          1 &          1 &          1 &          0 &          1 \\

 $\mathbf{s_8}$  &        1 &          1 &          1 &          1 &          1 &          1 &          0 &          0 \\
 
\hline
\end{tabular}
}
}
\subtable[PROMETHEE I]{%
\resizebox{0.25\textwidth}{!}{
\begin{tabular}{ccccccccc}
\hline
    &  $\mathbf{s_1}$  &  $\mathbf{s_2}$ &  $\mathbf{s_3}$ &  $\mathbf{s_4}$ &  $\mathbf{s_5}$ &  $\mathbf{s_6}$ &  $\mathbf{s_7}$ &  $\mathbf{s_8}$ \\ 
\hline
 $\mathbf{s_1}$  &  0 &          1 &          1 &          1 &          1 &          1 &          0 &          1 \\

 $\mathbf{s_2}$  &        0 &          0 &          0 &          1 &          0 &          1 &          0 &          0 \\

 $\mathbf{s_3}$  &        1 &          1 &          0 &          1 &          1 &          1 &          1 &          1 \\

 $\mathbf{s_4}$  &        1 &          1 &          0 &          0 &          1 &          1 &          1 &          1 \\
 
 $\mathbf{s_5}$  &        1 &          1 &          1 &          1 &          0 &          1 &          0 &          1 \\

 $\mathbf{s_6}$  &        1 &          1 &          1 &          1 &          0 &          0 &          0 &          0 \\

 $\mathbf{s_7}$  &        1 &          1 &          1 &          1 &          1 &          1 &          0 &          1 \\

 $\mathbf{s_8}$  &        1 &          1 &          1 &          1 &          1 &          1 &          0 &          0 \\      
\hline
\end{tabular}
}
}
\end{center}
\end{table}

\noindent Looking at Tables \ref{first_necessary}, we underline that $s_{7}$, $s_3$ and $s_{5}$ are surely the best among the eight students considered. In fact, $s_{7}$ is necessarily preferred to five out of the other seven students both locally and considering the bipolar PROMETHEE II method and, at the same time, (s)he is the only student being necessarily preferred to some other student using the bipolar PROMETHEE I method. $s_{3}$ is necessarily preferred to four out of the other seven students locally, and (s)he is necessarily preferred to $s_4$ considering the bipolar PROMETHEE II method. At the same time, (s)he is locally possibly preferred to $s_7$ (see Table \ref{poss_first}). $s_{5}$ is necessarily preferred to $s_2$ and $s_{6}$ considering the bipolar PROMETHEE II method. In order to get a more insight on the problem at hand, we suggest to the Dean to provide other information (s)he is sure about. For this reason, the Dean states that, locally, $s_{2}$ is preferred to $s_6$ and $s_{8}$ is preferred to $s_{1}$.

\begin{table}[!h]
\begin{center}
\caption{Necessary preference relations after the second piece of preference information\label{nec_second}}
\subtable[Local]{%
\resizebox{0.25\textwidth}{!}{
\begin{tabular}{ccccccccc}
\hline
    &  $\mathbf{s_1}$  &  $\mathbf{s_2}$ &  $\mathbf{s_3}$ &  $\mathbf{s_4}$ &  $\mathbf{s_5}$ &  $\mathbf{s_6}$ &  $\mathbf{s_7}$ &  $\mathbf{s_8}$ \\ 
\hline
 $\mathbf{s_1}$  &         0 &          1 &          0 &          0 &          0 &          1 &          0 &          0 \\

 $\mathbf{s_2}$  &         0 &          0 &          0 &          0 &          0 &          \cellcolor{yellow}{1} &          0 &          0 \\

 $\mathbf{s_3}$  &         1 &          1 &          0 &          1 &          \cellcolor{yellow}{1} &          1 &          0 &          \cellcolor{yellow}{1} \\

 $\mathbf{s_4}$  &         \cellcolor{yellow}{1} &          1 &          0 &          0 &          0 &          0 &          0 &          0 \\

 $\mathbf{s_5}$  &         0 &          1 &          0 &          0 &          0 &          1 &          0 &          0 \\

 $\mathbf{s_6}$  &         0 &          0 &          0 &          0 &          0 &          0 &          0 &          0 \\

 $\mathbf{s_7}$  &         1 &          1 &          0 &          \cellcolor{yellow}{1} &          1 &          1 &          0 &          1 \\

 $\mathbf{s_8}$  &         \cellcolor{yellow}{1} &          1 &          0 &          0 &          1 &          1 &          0 &          0 \\

\hline
\end{tabular}
}
}
\subtable[PROMETHEE II]{%
\resizebox{0.25\textwidth}{!}{
\begin{tabular}{ccccccccc}
\hline
    &  $\mathbf{s_1}$  &  $\mathbf{s_2}$ &  $\mathbf{s_3}$ &  $\mathbf{s_4}$ &  $\mathbf{s_5}$ &  $\mathbf{s_6}$ &  $\mathbf{s_7}$ &  $\mathbf{s_8}$ \\ 
\hline
 $\mathbf{s_1}$  &   0 &          0 &          0 &          0 &          0 &          0 &          0 &          0 \\

 $\mathbf{s_2}$  &        0 &          0 &          0 &          0 &          0 &          0 &          0 &          0 \\

 $\mathbf{s_3}$  &        0 &          0 &          0 &          1 &          0 &          0 &          0 &          0 \\

 $\mathbf{s_4}$  &        0 &          0 &          0 &          0 &          0 &          0 &          0 &          0 \\

 $\mathbf{s_5}$  &        0 &          1 &          0 &          0 &          0 &          1 &          0 &          0 \\

 $\mathbf{s_6}$  &        0 &          0 &          0 &          0 &          0 &          0 &          0 &          0 \\
 
 $\mathbf{s_7}$  &        1 &          1 &          0 &          \cellcolor{yellow}{1} &          1 &          1 &          0 &          1 \\

 $\mathbf{s_8}$  &        0 &          0 &          0 &          0 &          0 &          0 &          0 &          0 \\     
 
\hline
\end{tabular}
}
}
\subtable[PROMETHEE I]{%
\resizebox{0.25\textwidth}{!}{
\begin{tabular}{ccccccccc}
\hline
    &  $\mathbf{s_1}$  &  $\mathbf{s_2}$ &  $\mathbf{s_3}$ &  $\mathbf{s_4}$ &  $\mathbf{s_5}$ &  $\mathbf{s_6}$ &  $\mathbf{s_7}$ &  $\mathbf{s_8}$ \\ 
\hline
 $\mathbf{s_1}$  &        0 &          0 &          0 &          0 &          0 &          0 &          0 &          0 \\

 $\mathbf{s_2}$  &        0 &          0 &          0 &          0 &          0 &          0 &          0 &          0 \\

 $\mathbf{s_3}$  &        0 &          0 &          0 &          0 &          0 &          0 &          0 &          0 \\

 $\mathbf{s_4}$  &        0 &          0 &          0 &          0 &          0 &          0 &          0 &          0 \\

 $\mathbf{s_5}$  &        0 &          0 &          0 &          0 &          0 &          0 &          0 &          0 \\

 $\mathbf{s_6}$  &        0 &          0 &          0 &          0 &          0 &          0 &          0 &          0 \\

 $\mathbf{s_7}$  &        1 &          1 &          0 &          \cellcolor{yellow}{1} &          0 &          0 &          0 &          \cellcolor{yellow}{1} \\

 $\mathbf{s_8}$  &        0 &          0 &          0 &          0 &          0 &          0 &          0 &          0 \\
\hline
\end{tabular}
}
}
\end{center}
\end{table}

\begin{table}[!h]
\begin{center}
\caption{Possible preference relations after the second piece of preference information\label{pos_second}}
\subtable[Local]{%
\resizebox{0.25\textwidth}{!}{
\begin{tabular}{ccccccccc}
\hline
    &  $\mathbf{s_1}$  &  $\mathbf{s_2}$ &  $\mathbf{s_3}$ &  $\mathbf{s_4}$ &  $\mathbf{s_5}$ &  $\mathbf{s_6}$ &  $\mathbf{s_7}$ &  $\mathbf{s_8}$ \\ 
\hline
 $\mathbf{s_1}$  &        0 &          1 &          0 &          \cellcolor{yellow}{0} &          1 &          1 &          0 &          \cellcolor{yellow}{0} \\

 $\mathbf{s_2}$  &        0 &          0 &          0 &          0 &          0 &          1 &          0 &          0 \\

 $\mathbf{s_3}$  &        1 &          1 &          0 &          1 &          1 &          1 &          1 &          1 \\

 $\mathbf{s_4}$  &        1 &          1 &          0 &          0 &          1 &          1 &          \cellcolor{yellow}{0} &          1 \\

 $\mathbf{s_5}$  &        1 &          1 &          \cellcolor{yellow}{0} &          1 &          0 &          1 &          0 &          0 \\

 $\mathbf{s_6}$  &        0 &          \cellcolor{yellow}{0} &          0 &          1 &          0 &          0 &          0 &          0 \\

 $\mathbf{s_7}$  &        1 &          1 &          1 &          1 &          1 &          1 &          0 &          1 \\

 $\mathbf{s_8}$  &        1 &          1 &          \cellcolor{yellow}{0} &          1 &          1 &          1 &          0 &          0 \\   

\hline
\end{tabular}
}
}
\subtable[PROMETHEE II]{%
\resizebox{0.25\textwidth}{!}{
\begin{tabular}{ccccccccc}
\hline
    &  $\mathbf{s_1}$  &  $\mathbf{s_2}$ &  $\mathbf{s_3}$ &  $\mathbf{s_4}$ &  $\mathbf{s_5}$ &  $\mathbf{s_6}$ &  $\mathbf{s_7}$ &  $\mathbf{s_8}$ \\ 
\hline  
 $\mathbf{s_1}$  &        0 &          1 &          1 &          1 &          1 &          1 &          0 &          1 \\

 $\mathbf{s_2}$  &        1 &          0 &          1 &          1 &          0 &          1 &          0 &          1 \\

 $\mathbf{s_3}$  &        1 &          1 &          0 &          1 &          1 &          1 &          1 &          1 \\

 $\mathbf{s_4}$  &        1 &          1 &          0 &          0 &          1 &          1 &          \cellcolor{yellow}{0} &          1 \\

 $\mathbf{s_5}$  &        1 &          1 &          1 &          1 &          0 &          1 &          0 &          1 \\

 $\mathbf{s_6}$  &        1 &          1 &          1 &          1 &          0 &          0 &          0 &          1 \\

 $\mathbf{s_7}$  &        1 &          1 &          1 &          1 &          1 &          1 &          0 &          1 \\

 $\mathbf{s_8}$  &        1 &          1 &          1 &          1 &          1 &          1 &          0 &          0 \\
 
\hline
\end{tabular}
}
}
\subtable[PROMETHEE I]{%
\resizebox{0.25\textwidth}{!}{
\begin{tabular}{ccccccccc}
\hline
    &  $\mathbf{s_1}$  &  $\mathbf{s_2}$ &  $\mathbf{s_3}$ &  $\mathbf{s_4}$ &  $\mathbf{s_5}$ &  $\mathbf{s_6}$ &  $\mathbf{s_7}$ &  $\mathbf{s_8}$ \\ 
\hline
 $\mathbf{s_1}$  &        0 &          1 &          \cellcolor{yellow}{0} &          1 &          1 &          1 &          0 &          1 \\

 $\mathbf{s_2}$  &        0 &          0 &          0 &          1 &          0 &          1 &          0 &          0 \\

 $\mathbf{s_3}$  &        1 &          1 &          0 &          1 &          1 &          1 &          1 &          1 \\

 $\mathbf{s_4}$  &        1 &          1 &          0 &          0 &          1 &          1 &          \cellcolor{yellow}{0} &          1 \\

 $\mathbf{s_5}$  &        1 &          1 &          1 &          1 &          0 &          1 &          0 &          1 \\

 $\mathbf{s_6}$  &        \cellcolor{yellow}{0} &          1 &          1 &          1 &          0 &          0 &          0 &          0 \\

 $\mathbf{s_7}$  &        1 &          1 &          1 &          1 &          1 &          1 &          0 &          1 \\

 $\mathbf{s_8}$  &        1 &          1 &          1 &          1 &          1 &          1 &          0 &          0 \\     
\hline
\end{tabular}
}
}
\end{center}
\end{table}

\noindent Translating these preference information using the constraints $Ch^{B}(P^{B}(2,6),\hat\mu)>0$ and $Ch^{B}(P^{B}(8,1),\hat\mu)>0$, and computing again the necessary and possible preference relations locally and considering both the bipolar PROMETHEE methods, we get the results shown in Tables \ref{nec_second} and \ref{pos_second}. In these Tables, yellow cells correspond to new information we have got using the second piece of information provided by the Dean. In particular, in Tables \ref{nec_second} the cell in correspondence of the pair of students $(s_i,s_j)$ is yellow colored if $s_i$ was not necessarily preferred to $s_j$ after the first iteration, but $s_{i}$ is necessarily preferred to $s_{j}$ after the second iteration; in Tables \ref{pos_second}, the cell in correspondence of the pair of students $(s_i,s_j)$ is yellow colored if $s_i$ was possibly preferred to $s_{j}$ after the first iteration but $s_{i}$ is not possibly preferred to $s_j$ after the second iteration anymore. Looking at Tables \ref{nec_second} and \ref{pos_second}, the Dean is addressed to consider $s_7$ as the best student. In fact, also if $s_7$ and $s_3$ are locally necessarily preferred to all other six considered students, $s_{7}$ is still the only one being necessarily preferred to someone else considering the bipolar PROMETHEE I method. Besides, looking at Tables \ref{pos_second}, we get that $s_{3}$ is the only student being possibly preferred to $s_{7}$ locally and with respect to PROMETHEE I and PROMETHEE II but, at the same time, everyone except $s_4$, is possibly preferred to $s_{3}$ considering the bipolar PROMETHEE I method while four students ($s_5$, $s_6$, $s_7$ and $s_8$) are possibly preferred to $s_{3}$ with respect to the bipolar PROMETHEE I method.

\section{Conclusions}
In this paper we proposed a generalization of the classical PROMETHEE methods. A basic assumption of PROMETHEE methods is the independence between criteria which implies that no interaction between criteria is considered. In this paper we developed a methodology permitting to take into account interaction between criteria (synergy, redundancy and antagonism effects) within PROMETHEE method by using the bipolar Choquet integral. In this way we obtained a new method called the Bipolar PROMETHEE method.\\
The Decision Maker (DM) can give directly the preferential parameters of the method; however, due to their great number, it is advisable using some indirect procedure to elicit the preferential parameters from some preference information provided by the DM. \\
Since, in general, there is more than one set of parameters compatible with these preference information, we proposed to use the Robust Ordinal Regression (ROR) to consider the whole family of compatible sets of preferential parameters. We believe that the proposed methodology can be successfully applied in many real world problems where interacting criteria have to be considered; besides, in a companion paper, we propose to apply the SMAA methodology to the classical and to the bipolar PROMETHEE methods (for a survey on SMAA methods see \cite{tervonen_figueira}).

\bibliographystyle{plain}
\bibliography{Bib_Integral}

\end{document}